\documentclass{amsart}
\usepackage{amscd,amsmath,amssymb}
\usepackage{pstricks}
\usepackage{latexsym,amsbsy,mathrsfs}
\usepackage{xy}
\usepackage{xypic}
\usepackage{pgf,tikz}

\newtheorem{thm}{Theorem}[section]
\newtheorem{lem}[thm]{Lemma}
\newtheorem{cor}[thm]{Corollary}
\newtheorem{prop}[thm]{Proposition}

\theoremstyle{definition}
\newtheorem{example}[thm]{Example}

\newtheorem{defn}[thm]{Definition}

\newtheorem{conj}[thm]{Conjecture}
\newtheorem{rem}[thm]{Remark}

\numberwithin{equation}{thm}

\newcommand{\rank}{\underline{\operatorname{rank}}}
\newcommand{\dl}{\operatorname{dl}}

\normalbaselines
\begin{document}
\title[ AR quivers of string algebras of type $\widetilde{C}$ and a conjecture of GLS]
{The Auslander-Reiten quivers of string algebras of affine type $\widetilde{C}$ and a conjecture by Geiss-Leclerc-Schr\"{o}er}

\author{Hua-lin Huang, Zengqiang Lin$^*$ and Xiuping Su}
\address{Hua-lin Huang, Zengqiang Lin: School of Mathematical sciences, Huaqiao University,
Quanzhou\quad 362021,  China.}
\address{Xiuping Su: Department of Mathematical Sciences, University of Bath, Bath BA2 7JY, UK}

\thanks{The first author was supported by  the National Natural Science Foundation of China (Grants No. 11911530172 and 11971181). The second author was supported by the Natural Science Foundation of Fujian Province (Grant No. 2020J01075)}
\thanks{$^*$ The corresponding author.}
\thanks{Email: hualin.huang@foxmail.com; zqlin@hqu.edu.cn; xs214@bath.ac.uk.}


\begin{abstract}
In this paper, we study representations of certain string algebras, which are referred to as of affine type $\widetilde{C}$. We introduce minimal string modules and apply them to explicitly describe components of the
Auslander-Reiten quivers of the string algebras and $\tau$-locally free modules defined by Geiss-Lerclerc-Schr\"{o}er. As an application, we prove
 Geiss-Leclerc-Schr\"{o}er's conjecture on the correspondence between positive roots of type $\widetilde{C}$ and
$\tau$-locally free modules of the corresponding string algebras.
\end{abstract}

\maketitle

\section{Introduction}

Given  a symmetrizable Cartan matrix $C$ with a symmetrizer $D$,
Geiss-Leclerc-Schr\"{o}er  \cite{[GLS1]} construct a  quiver $Q=Q(C, \Omega)$ and define
a quotient path algebra $H=H(C,D,\Omega)=KQ/I$, where $K$ is a field and $I$ is an ideal generated by some nilpotency relations and some commutative relations.
In particular, there is a loop at each vertex in $Q$ and powers of the nilpotency relations in $I$ encode the symmetrizer $D$.
They then develop a sequence of work based on the representation theory of $H$ \cite{[GLS1], [GLS2], [GLS3], [GLS4], [GLS5]}, providing a uniform approach to studying connections between
 representation theory of simply laced and non-simply laced (or valued) quivers, and Lie theory and cluster theory. For instance, it includes
a generalisation of two fundamental results in quiver representation theory,  Gabriel's Theorem and Dlab-Ringel's Theorem (to be made more precise later), a  construction of enveloping algebras and a generalisation of Caldero-Chapoton's formula in cluster theory.

We are interested in the aspect of the correspondence between $\tau$-locally free  $H$-modules and positive roots of type $C$ \cite{[GLS1], [GLS2]}.
Let  $e_i$ be the idempotent in $H$ corresponding to the vertex $i$ in $Q$ and $H_i=e_iHe_i$. A left $H$-module $M$ is said to be  {\em locally free} if $M_i=e_iM$ is a free $H_i$-module for all $i$, and for such a module $M$, denote by $\underline{\text{rank}}(M)=(r_1,\cdots,r_n)$ the rank vector of $M$. That is,  each $r_i$ is  the rank of the free $H_i$-module $M_i$.
 An indecomposable $H$-module $M$ is  {\em $\tau$-locally free} if the AR-translations $\tau^k(M)$ for all $k\in \mathbb{Z}$ are locally free. Note that not all indecomposable locally free modules are $\tau$-locally free.
Geiss-Leclerc-Schr\"{o}er \cite{[GLS1]} prove that  there are only finitely many isomorphism classes of $\tau$-locally free $H$-modules if and only if the Cartan matrix $C$ is of Dynkin type. Moreover, in this case, the assignment $M\mapsto \underline{\text{rank}}(M)$ offers a bijection between the isomorphism classes of  $\tau$-locally free $H$-modules and
the positive roots of type $C$, i.e. the positive roots of a complex Lie algebra defined by $C$.  These results generalize Gabriel's  Theorem for Dynkin quivers \cite{[Gab]} and Dlab-Ringel's  Theorem for Dynkin (valued) quivers \cite{[DR]} (also known as modulated graphs, see for instance \cite{[GLS1]}),
in the sense that both theorems provide a one to one correspondence between the isomorphism classes of representations of a Dynkin (valued) quiver of type $C$ and the positive roots of type $C$ via the map sending an indecomposable representation to its dimension vector.

For non-Dynkin symmetrizable Cartan matrices,  Geiss-Leclerc-Schr\"{o}er propose the following conjecture
\cite[Conjecture 5.3]{[GLS2]}.

\vspace{2mm}\noindent
{\bf{Conjecture 1}} [Geiss-Leclerc-Schr\"{o}er]
{\em Let $H=H(C,D,\Omega)$. Then
there is a bijection between  the positive roots of type $C$ and  the rank vectors of $\tau$-locally free $H$-modules.}

\vspace{2mm}

Evidence supporting the conjecture includes the following. First, when $C$ is symmetric and $D$ is the identity matrix, the conjecture is true by Kac's Theorem \cite{[K1],[K2]}. Second,
Geiss-Leclerc-Schr\"{o}er  \cite{[GLS6]} prove that there is a bijection between  isomorphism classes of $\tau$-locally free  rigid $H$-modules and  real Schur roots of $Q$. Note also Chen-Wang \cite{[CW]} work on categorification of foldings of root lattices and in the case of Dynkin type they recover
Geiss-Leclerc-Schr\"{o}er's result on the correspondence between $\tau$-locally free $H$-modules and positive roots of type $C$.

In general, Conjecture 1 is still open.  In this paper, we consider an affine case of type $\widetilde{C}_{n-1}$, that is,  the Cartan matrix $C$ is the following $n\times n$ matrix
 $$C=\left(
                \begin{array}{ccccccc}
                  2 & -1 &  &  & & \\
                  -2 & 2 & -1 &  &  &\\
                  & -1 & 2 & -1 & & \\
                 & & \ddots & \ddots & \ddots & \\
                &  & & -1 & 2 & -1 & \\
                 & & &  & -1 & 2 & -2 \\
                 & &  &  &  & -1 & 2 \\
                \end{array}
              \right)$$
and the symmetrizer  $D=\textup{diag}(2,1,1,\cdots,1,1,2)$.  {
Then the algebra $H=H(C,D,\Omega)$ $=KQ/I$, where
$Q$ is a quiver of type $A_n$ when the two loops at 1 and $n$ are ignored},
$$ \xymatrix{
\varepsilon_1 \circlearrowleft  1 \ar@{-}[r] & 2\ar@{-}[r]&  \cdots \ar@{-}[r]  &n\circlearrowleft \varepsilon_n\\
}$$
and $I$ is  the  ideal generated by
$\varepsilon_1^2$ and $\varepsilon_n^2$. In particular, $H$ is a string algebra and we say that $H$
is a  {\em string algebra of type} $\widetilde{C}_{n-1}$. Note that, if $C$ is of other affine type, then  $H=H(C,D,\Omega)$ is not a string algebra.
We will study the representation category of $H$, in particular the Auslander-Reiten theory of $H$, using techniques from string algebras,  and minimal string modules that are  to be introduced later in this paper.
The explicit construction of Auslander-Reiten  sequences (also written as AR-sequences) for string algebras in  \cite{[DR]}  by Butler-Ringel  is particularly helpful in our understanding of the Auslander-Reiten quiver (also written as AR-quiver) of $H$.

We define the {\em index} of an indecomposable $H$-module $M$ to be $(a,b)$, where  $a$ is the number of   irreducible maps to $M$  and $b$ is the number of irreducible maps from $M$ in the AR-quiver of $H$ and we say  that a string module $M$ is {\em minimal} if if each irreducible map from $M(w)$ is injective and each irreducible map to $M(w)$ is surjective. Using Butler-Ringel's construction of AR-sequences,  we classify all the minimal string modules. This classification
then leads to explicit description of connected components of the AR-quiver of $H$. Consequently, we know precisely where $\tau$-locally free modules are in the AR-quiver and so prove Conjecture 1 for the case where $C$ is of type $\widetilde{C}_{n-1}$ and $D$ is minimal. We have the following main results.

\vspace{1mm}\noindent
{\bf Theorem A} (Theorem \ref{thm3.1}) {\em Let $H=H(C,D,\Omega)$ be a string algebra of type $\widetilde{C}_{n-1}$.
 The AR-quiver $\Gamma_H$ of $H$ consists of the following, which are pairwise disjoint.
\begin{itemize}
\item[(1)] One component $\mathcal{T}_{PI}$ containing all the indecomposable preprojective modules and all the indecomposable preinjective modules (up to isomorphism).

\item[(2)] One tube of rank $n-1$.

\item[(3)] Homogeneous tubes $\mathcal{H}_{w,S}$, where $w$ runs through all the representatives of bands in $H$ and
$S$ runs through the isomorphism classes of simple modules of the Laurent polynomial ring $K[T,T^{-1}]$.

\item[(4)] Components $\mathcal{T}_\lambda$ of type $\mathbb{Z}A_\infty^\infty$,  where $\lambda$ runs through the isomorphism classes of minimal string modules of type (2,2).
\end{itemize}}

\vspace{1mm}\noindent
{\bf Theorem B} (Theorem \ref{thm3.2})
{\em Let $H=H(C,D,\Omega)$ be a string algebra of type $\widetilde{C}_{n-1}$ and
let $M$ be an indecomposable $H$-module. Then $M$ is $\tau$-locally free if and only if one of the following is satisfied:
\begin{itemize}
\item[(1)] $M$ is preprojective.

\item[(2)] $M$ is preinjective.

\item[(3)] $M$ is a regular module occurring in any tube.
\end{itemize}}

\vspace{1mm}\noindent
{\bf Theorem C} (Theorem \ref{thm4.1}) {
{\em Let $H=H(C,D,\Omega)$ be a string algebra of type $\widetilde{C}_{n-1}$.
Then $\alpha$ is a positive root of type $C$ if and only if there is a
$\tau$-locally free module $M$ such that $\rank M=\alpha$. Moreover,
\begin{itemize}
\item[(1)] if $\alpha$ is a positive real root, then there is a unique $\tau$-locally free $H$-module $M$ (up to isomorphism) such that $\underline{\textup{rank}} M =\alpha$.

\item[(2)] if $\alpha$ is a positive imaginary root,
then there are families of  $\tau$-locally free $H$-modules $M$ such that  $\underline{\textup{rank}}(M)=\alpha$.
\item[(3)] the modules at the bottom of the tube of rank $n-1$ are rigid.
\end{itemize}}}

\vspace{1mm}\noindent
{\bf Corollary D} (Corollary \ref{cor4.4}) {\em Let $C$ be a Cartan matrix of type $\widetilde{C}_{n-1}$ and $D=\textup{diag}(2,1,\cdots, \,1,2)$. Then Conjecture 1 is true.}

\vspace{1mm}

The remaining part of this paper is organized as follows. In Section 2, we recall some basic definitions on string algebras and Butler-Ringel's construction of AR-sequences. In Section 3, we develop the theory of minimal string modules to prove Theorem A and Theorem B. In Section 4, we first recall basic definitions and facts on root systems and Weyl groups,
and then prove Theorem C and Corollary D.

\section{Basic notions and facts on string algebras}

In this section, we recall the definition of string algebras and basic properties of their module categories \cite{[BR]}.
Let $K$ be a field and $A$ be a finite dimensional $K$-algebra. Throughout this paper, all modules are left $A$-modules. We denote by $S_1,S_2,\cdots,S_n$ a complete list of simple $A$-modules,  and by $P_1,P_2,\cdots,P_n$ (resp. $I_1,I_2,\cdots,I_n$) a complete list of  indecomposable projective (resp. injective) $A$-modules (up to isomorphism).

\subsection{String algebras}

Let $Q=(Q_0,Q_1)$ be a quiver, where $Q_0$ denotes the set of vertices and $Q_1$ denotes the set of arrows in $Q$. Given an arrow $\alpha\in Q_1$, its starting and ending vertices are denoted by $s(\alpha)$ and $t(\alpha)$, respectively.

\begin{defn}
A finite dimensional $K$-algebra $A=KQ/I$ is called a {\em string algebra} if  the following conditions are satisfied:

(1) for any vertex $i\in Q_0$, there are at most two incoming and at most two outgoing arrows.

(2) for any arrow $\alpha\in Q_1$, there is at most one arrow $\beta$ and at most one arrow $\gamma$ such that $\beta\alpha\notin I$ and $\alpha\gamma\notin I$.

(3) the ideal $I$ is generated by a set of zero relations.

In particular, a string algebra $A=KQ/I$ is called {\em gentle} if $I$ is generated by paths of length 2.
\end{defn}

\begin{example}\label{ex1}
Let $A=KQ/I$, where
$$Q:\varepsilon_1 \circlearrowleft  1\xrightarrow{\alpha_1} 2 \xrightarrow{\alpha_2} 3\circlearrowleft \varepsilon_3$$
and $I=<\varepsilon_1^2,\varepsilon_3^2>$. Then $A$ is a string algebra.
\end{example}

\subsection{Strings and bands}

Let $A=KQ/I$ be a string algebra. Given an arrow $\alpha\in Q_1$, we denote by $\alpha^{-1}$ the {\em formal inverse} of $\alpha$, with $s(\alpha^{-1})=t(\alpha)$ and $t(\alpha^{-1})=s(\alpha)$, and write $(\alpha^{-1})^{-1}=\alpha$. A {\em word} $w=c_1c_2\cdots c_m$ of length $m\geq 1$ is a sequence of arrows and their formal inverses such that $s(c_i)=t(c_{i+1})$ for $1\leq i<m$. We define $w^{-1}=(c_1c_2\cdots c_m)^{-1}=c_m^{-1}\cdots c_2^{-1}c_1^{-1}$, $s(w)=s(c_m)$ and $t(w)=t(c_1)$. A  word $w=c_1c_2\cdots c_m$ of length $m\geq 1$ is called a {\em string} if   $c_{i+1}\neq c_i^{-1}$, and no subword nor its inverse belongs to $I$. In addition, we associate two {\em trivial} strings $1_{(u,1)}$ and $1_{(u,-1)}$ of length zero for any vertex $u\in Q_0$, where $s(1_{(u,i)})=t(1_{(u,i)})=u$ and $(1_{(u,i)})^{-1}=1_{(u,-i)}$ for $i=1,-1$. A string $w=c_1c_2\cdots c_m$ is said to be {\em direct} if all the $c_i$ are arrows, and {\em inverse} if all the $c_i$ are inverses of arrows. By definition, a vertex is  both direct and inverse.
We denote by $\textup{St}(A)$ the set of all strings in $A$.

A nontrivial string $w$  is called a {\em band} if $s(w)=t(w)$ and each power $w^r$ is a string, but $w$ itself is not a power of a string of smaller length. We denote by $\textup{Ba}(A)$ the set of all bands in $A$.

On $\textup{St}(A)$, let $\rho$ be  the equivalence relation that  identifies every string $w$ with its inverse $w^{-1}$. On $\textup{Ba}(A)$, let $\rho'$ be the equivalence relation that identifies every string $w=c_1c_2\cdots c_m$ with any cyclically permuted strings $w_{(i)}=c_ic_{i+1}\cdots c_mc_1\cdots c_{i-1}$ and their inverses $w_{(i)}^{-1}$, $1\leq i\leq m$. We choose a complete set $\overline{\textup{St}}(A)$ of representatives of $\textup{St}(A)$ relative to $\rho$, and a complete set $\overline{\textup{Ba}}(A)$ of representatives of $\textup{Ba}(A)$ relative to $\rho'$.

We write $u\sim w$ if two strings (resp. bands) $u$ and $w$ are equivalent, and  $u\not\sim w$ otherwise.
Represent a string $w=\alpha_1^{\epsilon_1}\alpha_2^{\epsilon_2}\cdots\alpha_m^{\epsilon_m}$, where $\alpha_i\in Q_1$ and $\epsilon_i\in\{1,-1\}$ for all $i$, as a walk
$$\xymatrix{x_{1}\ar@{-}[r]^{\alpha_1}& x_2\ar@{-}[r]^{\alpha_{2}}&\cdots\ar@{-}[r]^{\alpha_m} & x_{m+1},
}$$
 where $x_1,x_2,\cdots,x_{m+1}$ are the vertices of $Q$ visited by $w$, $\alpha_i$ is an arrow from $x_{i+1}$ to $x_{i}$ if $\epsilon_i=1$, or an arrow from $x_i$ to $x_{i+1}$ if $\epsilon_i=-1$.
This equivalence relation induces an equivalence relation on the walks. That is,  the walk
  $$w:\xymatrix{x_{1}\ar@{-}[r]^{\alpha_1}& x_2\ar@{-}[r]^{\alpha_{2}}&\cdots\ar@{-}[r]^{\alpha_m} & x_{m+1}
}$$ is equivalent to the walk
$$w^{-1}:\xymatrix{x_{m+1}\ar@{-}[r]^{\alpha_m}& x_m\ar@{-}[r]^{\alpha_{m-1}}&\cdots\ar@{-}[r]^{\alpha_1} & x_{1}.
}$$
Similarly, walks of bands are equivalent if the corresponding bands are equivalent with respect to $\rho'$.

\begin{example}
Let $A$ be a string algebra as in Example \ref{ex1}. Then $w_1=\alpha_1^{-1}\alpha_2^{-1}\varepsilon_3\alpha_2\alpha_1$ is a string but not a band, and $w_2=\varepsilon_1\alpha_1^{-1}\alpha_2^{-1}\varepsilon_3\alpha_2\alpha_1=\varepsilon_1w_1$ is a band.
\end{example}

\subsection{String modules and band modules}
Let $w=\alpha_1^{\epsilon_1}\alpha_2^{\epsilon_2}\cdots\alpha_m^{\epsilon_m}$ be a string with
the corresponding walk $$\xymatrix{x_{1}\ar@{-}[r]^{\alpha_1}& x_2\ar@{-}[r]^{\alpha_{2}}&\cdots\ar@{-}[r]^{\alpha_{m-1}}& x_{n}\ar@{-}[r]^{\alpha_{m}} & x_{m+1}
}.$$
The {\em string module} defined by $w$ is the  representation $M(w)=((V_i)_{i\in Q_0},(\varphi_\alpha)_{\alpha\in Q_1})$, where the vector spaces
$$V_i=\left\{
        \begin{array}{ll}
          \oplus_{x_j=i}Kx_j & \textup{if}\ i=x_j\ \textup{for some}\ j\in\{1,2,\cdots,m+1\}, \\
          0 & \textup{otherwise}, \\
        \end{array}
      \right.
$$
and the linear maps $\varphi_\alpha$ are given by $$\varphi_\alpha(x_{s(\alpha_k)})=\left\{
        \begin{array}{ll}
          x_{t(\alpha_k)} & \textup{if}\ \alpha=\alpha_{k}\ \textup{for\ some}\ 1\leq k\leq m, \\
          0 & \textup{otherwise}. \\
        \end{array}
      \right.$$
The module $M(w)$ can be unfolded as a representation $U(w)$ as follows,
$$\xymatrix{U_{x_{1}}\ar@{-}[r]^{U_{\alpha_1}}& U_{x_2}\ar@{-}[r]^{U_{\alpha_{2}}}&\cdots\ar@{-}[r]^{U_{\alpha_{m-1}}}& U_{x_{m}}\ar@{-}[r]^{U_{\alpha_{m}}} & U_{x_{m+1}}
}$$
where $U_{x_i}=K$  and $U_{\alpha_j}=\textup{id}_K$ for all $i$ and $j$.

By construction, $$\textup{dim}_KV_i=|\{j\in\{1,2,\cdots,m+1\}|x_j=i\}|$$ for any $i\in Q_0$, and $M(w)\cong M(w^{-1})$ as $A$-modules for any string $w$, and $M(1_{(u,t)})$ is the simple representation corresponding to the vertex $u$.

Next we explain the construction of a band module. Let $w=\alpha_1^{\epsilon_1}\alpha_2^{\epsilon_2}\cdots\alpha_m^{\epsilon_m}$
be a band with the corresponding walk
$$\xymatrix{
x_{1}\ar@{-}[r]^{\alpha_1}\ar@{-}[rrdd]_{\alpha_m} & x_2\ar@{-}[r]^{\alpha_{2}}&\cdots\ar@{-}[r]^{\alpha_{m-3}} &
x_{m-2}\ar@{-}[r]^{\alpha_{m-2}}& x_{m-1}\ar@{-}[lldd]^{\alpha_{m-1}}\\
&&&&&\\
&&x_m.&&&}$$
{Let  $X$ be a  module of the Laurent polynomial ring  $K[T,T^{-1}]$. Then $X$ is determined
by $s=\textup{dim} X$ and an  automorphism $\varphi$ of $X=K^s$.  So we also write $X=(K^s, \varphi)$.
Let $U(w, s ,  \varphi)$ be the representation associated to the walk $w$ and the module $X$ as follows},
$$\xymatrix{
U_{x_{1}}\ar@{-}[r]^{U_{\alpha_1}}\ar@{-}[rrdd]_{U_{\alpha_m}}& U_{x_2}\ar@{-}[r]^{U_{\alpha_{2}}}&\cdots\ar@{-}[r]^{U_{\alpha_{m-3}}} &
U_{x_{m-2}}\ar@{-}[r]^{U_{\alpha_{m-2}}}& U_{x_{m-1}}\ar@{-}[lldd]^{U_{\alpha_{m-1}}}\\
&&&&&\\
&&U_{x_m} &&&}$$
where $U_{x_{i}}=K^s$ for all $i=1,2,\cdots, m$ and
$$U_{\alpha_i}=\left\{
        \begin{array}{ll}
         \varphi & \textup{if}\ i=1\ \textup{and}\ \epsilon_1=1, \\
          \varphi^{-1} & \textup{if}\ i=1\ \textup{and}\ \epsilon_1=-1, \\
          \textup{id}_{K^s} & \textup{if}\ 2\leq i\leq m. \\
        \end{array}
      \right.$$
Now the  {\em band module} $M(w,s,\varphi)=((V_i)_{i\in Q_0},(\varphi_\alpha)_{\alpha\in Q_1})$ is defined by
$$V_i=\left\{
        \begin{array}{ll}
          \oplus_{x_j=i}U_{x_j} & \textup{if}\ i=x_j\ \textup{for some}\ j\in\{1,2,\cdots,m\}, \\
          0 & \textup{otherwise}, \\
        \end{array}
      \right.
$$ and $$\varphi_\alpha=\left\{
        \begin{array}{ll}
          \oplus_{\alpha_i=\alpha}U_{\alpha_i} & \textup{if}\ \alpha=\alpha_i\ \textup{for some}\ i\in\{1,2,\cdots,m\}, \\
          0 & \textup{otherwise}. \\
        \end{array}
      \right.
$$
From the definition, one can check that $M(w,s,\varphi)\cong M(w^{-1},s,\varphi^{-1})$ and $M(w,s,\varphi)\cong M(w',s,\varphi)$, where $w'$ is equivalent to $w$ with respect to $\rho'$.

\begin{example}
Let $A$ be a string algebra as in Example \ref{ex1}.

(1) For the string $w_1=\alpha_1^{-1}\alpha_2^{-1}\varepsilon_3\alpha_2\alpha_1$, the string module $M(w_1)$ is as follows.
$$\xymatrix{
0 \circlearrowleft  K^2\ar[r]^{\ \ \ \ \textup{id}}& K^2 \ar[r]^{\textup{id} \ \ \ \ } & K^2\circlearrowleft ^{\left(
                                                                              \begin{smallmatrix}
                                                                                0 & 0 \\
                                                                                1 & 0 \\
                                                                              \end{smallmatrix}
                                                                            \right)}
\\
}$$

(2) For the band $w_2=\varepsilon_1\alpha_1^{-1}\alpha_2^{-1}\varepsilon_3\alpha_2\alpha_1$, the band module $M(w_2,1,\lambda)$ is as follows, where $\lambda\not= 0.$
$$\xymatrix{
^{\left(
                                                                              \begin{smallmatrix}
                                                                                0 & \lambda \\
                                                                                0 & 0 \\
                                                                              \end{smallmatrix}
                                                                            \right)}
\circlearrowleft  K^2\ar[r]^{\ \ \ \ \ \textup{id} }& K^2 \ar[r]^{ \textup{id}\ \ \ \ \ }
& K^2\circlearrowleft ^{\left(
                                                                              \begin{smallmatrix}
                                                                                0 & 0 \\
                                                                                1 & 0 \\
                                                                              \end{smallmatrix}
                                                                            \right)}
\\
}$$
\end{example}

{ Denote by $\mathcal{M}$ a complete set of representatives of indecomposable $K[T, T^{-1}]$-modules.
\begin{thm}\cite[Theorem 3.1]{[BR]} \label{BRThm} Let $A$ be a string algebra.
Then the string modules $M(w)$ with $w\in\overline{\textup{St}}(A)$ and the band modules $M(w,s,\varphi)$ with $w\in\overline{\textup{Ba}}(A)$ and $(K^s,\varphi)\in \mathcal{M}$
are up to isomorphism all the indecomposable $A$-modules.
\end{thm}}

\subsection{Auslaner-Reiten sequences for string algebras}

For each arrow $\alpha\in Q_1$, let
\[\alpha_{-} =\beta_1^{-1}\beta_2^{-1}\cdots \beta_r^{-1}\]
be the inverse string of maximal length such that $\alpha\cdot\alpha_{-}$ is a string, and  let
\[{}_{-}\alpha=\gamma_s^{-1}\cdots\gamma_2^{-1}\gamma_1^{-1}\]
be the inverse string of maximal length such that
$_{-}\alpha\cdot\alpha$ is a string.
Similarly, let
 \[_{+}(\alpha^{-1})=\beta_r\cdots\beta_2\beta_1 ~ (\text{resp. } (\alpha^{-1})_{+}=\gamma_1\gamma_2\cdots \gamma_s)\] be the direct string of maximal length such that $_{+}(\alpha^{-1})\cdot\alpha^{-1}$ (resp. $\alpha^{-1}\cdot(\alpha^{-1})_{+}$) is a string.

\begin{prop}\label{prop1.1}\cite{[BR]}
The only AR-sequences that consist of string modules and that have the middle term indecomposable are
$$0\rightarrow M(_{-}\alpha)\rightarrow M(_{-}\alpha\cdot\alpha\cdot\alpha_{-})\rightarrow M(\alpha_{-})\rightarrow 0,$$
 where $\alpha\in Q_1$.
\end{prop}

Next we describe the AR-sequences with the middle term decomposable. We will see shortly that in this case,  the middle term in such a short exact sequence is a direct sum of two indecomposable modules.

\begin{defn}

(1) A string $w$ is {\em right directly extendable} (RDE) if there is an arrow $\alpha$ such that $w\alpha$ is a string.

(2) A string $w$ is {\em right inversely extendable} (RIE) if there is an  arrow $\beta$ such that $w\beta^{-1}$ is a string.

(3) A string $w$ is {\em left directly extendable} (LDE) if there is an  arrow $\alpha$ such that $\alpha w$ is a string.

(4) A string  $w$ is {\em left inversely extendable} (LIE) if there is an  arrow $\beta$ such that $\beta^{-1} w$ is a string.
\end{defn}

\begin{rem}
Comparing with the terminology in \cite{[BR]}, we have the following.

(1) A string $w$ is not RDE if and only if $w$ starts on a peak.

(2) A string $w$ is not RIE if and only if $w$ starts in a deep.

(3) A string $w$ is not LDE if and only if $w$ ends in a deep.

(4) A string $w$ is not LIE if and only if $w$ ends on a peak.
\end{rem}

If $w$ is RDE, then there exists an arrow $\alpha$ such that $w\alpha$ is a string. Let
$$\xymatrixcolsep{1.2pc}\xymatrixrowsep{0.6pc}\xymatrix{&&&&\bullet\ar[ld]_{\alpha}\ar[rrdd]^{(\alpha_{-})^{-1}}&&\\
w_h:=w\cdot\alpha\cdot\alpha_{-}=&\bullet \ar@{--}[rr]^{w}& &\bullet & & &\\
&&&&&&\bullet & }$$
We say $w_h$ is {\em obtained from $w$ by adding a hook from the right}. There is a canonical embedding $i:M(w)\rightarrow M(w_h)$.

{If
$w=u\cdot\beta^{-1}\cdot(\beta^{-1})_{+}$ for some string $u$ and some arrow $\beta$, $$\xymatrixcolsep{1.2pc}\xymatrixrowsep{0.6pc}\xymatrix{
&&&&&&\bullet\ar[lldd]^{(\beta^{-1})_{+}}\\
w=u\cdot\beta^{-1}\cdot(\beta^{-1})_{+}=& \bullet\ar@{--}[rr]^{u}&&\bullet\ar[rd]_{\beta}& & & \\
&&&&\bullet &&\\
}$$
then we say} $u$ is {\em obtained from $w$ by deleting a cohook from the right}.
In this case, $u$ is RIE and there is a canonical projection $p:M(w)\rightarrow M(u)$.
The string $w$ can be understood as being obtained from $u$ by adding a cohook from the right and so
we also write $w=u_c$.

If $w$ is LIE, then there exists an arrow $\alpha$ such that $\alpha^{-1}w$ is a string. Let
$$\xymatrixcolsep{1.2pc}\xymatrixrowsep{0.6pc}\xymatrix{
&&&\bullet\ar[lldd]_{_{+}(\alpha^{-1})}\ar[rd]^{\alpha}&&&&&&\\
_hw:=\ _{+}(\alpha^{-1})\cdot\alpha^{-1}\cdot w=&& && \bullet\ar@{--}[rr]^{w} &&\bullet\\
&\bullet&&&&&&&&
}$$
We say $_hw$ is {\em obtained from $w$ by adding a hook from the left}. There is a canonical embedding $i:M(w)\rightarrow M(_hw)$.

{If  $w={}_{-}\beta\cdot\beta\cdot u$ for some string $u$ and some arrow $\beta$,
$$\xymatrixcolsep{1.2pc}\xymatrixrowsep{0.6pc}\xymatrix{
&\bullet\ar[rrdd]_{(_{-}\beta)^{-1}}&&&&&&\\
w=\ _{-}\beta\cdot\beta\cdot u=&&&&\bullet\ar[ld]^{\beta}\ar@{--}[rr]^{u}&&\bullet &&\\
&&&\bullet&&&&
}$$
then we say} $u$ is {\em obtained from $w$ by deleting a cohook from the left}. In this case, $u$ is LDE and there is a canonical projection $p:M(w)\rightarrow M(u)$.
Similar to $u_c$ above, the string $w$ can be understood as being obtained from $u$ by adding a cohook from the left
and so we  also write
 $w={}_cu$.

\begin{prop} \cite{[BR]} \label{prop1.2}
Let $w$ be a string such that $M(w)$ is not injective and $w\not\sim{}_{-}\alpha$ for any $\alpha\in Q_1$.
\begin{itemize}
\item[(1)] If $w$ is RDE and LIE, then the following
$$0\rightarrow M(w)\xrightarrow{
\left(
                                                                         \begin{smallmatrix}
                                                                           i & i\\
                                                                         \end{smallmatrix}
                                                                       \right)} M(_hw)\oplus M(w_h)\xrightarrow{\left(
                                                                         \begin{smallmatrix}
                                                                           i \\
                                                                           -i \\
                                                                         \end{smallmatrix}
                                                                       \right)
} M(_hw_h)\rightarrow 0$$
is an AR-sequence where $_hw_h={}_h(w_h)=(_hw)_h$.

\item[(2)] If $w$ is RDE but not LIE, then $w={}_cu$ for some string $u$ and the following
$$0\rightarrow M(w)\xrightarrow{\left( \begin{smallmatrix}
                                        p & i\\
                                        \end{smallmatrix}
                                         \right)}
M(u)\oplus M(w_h)\xrightarrow{\left(\begin{smallmatrix}
                                         i \\
                                         -p \\
                                        \end{smallmatrix}
                                       \right)
} M(u_h)\rightarrow 0$$
is an AR-sequence.

\item[(3)] If $w$  is LIE but not RDE, then $w=u_c$ for some string $u$ and the following
$$0\rightarrow M(w)\xrightarrow{\left( \begin{smallmatrix}
                                        i & p\\
                                        \end{smallmatrix}
                                         \right)}
M(_hw)\oplus M(u)\xrightarrow{\left(\begin{smallmatrix}
                                         p \\
                                         -i \\
                                        \end{smallmatrix}
                                       \right)
} M(_hu)\rightarrow 0$$
is an AR-sequence.

\item[(4)] If $w$ is neither RDE nor LIE, then $w=$ $_cu_c={}_c(u_c)=({}_c u)_c$ for some string $u$ and the following
$$0\rightarrow M(w)\xrightarrow{\left( \begin{smallmatrix}
                                        p & p\\
                                        \end{smallmatrix}
                                         \right)}
M(u_c)\oplus M(_cu)\xrightarrow{\left(\begin{smallmatrix}
                                         p \\
                                         -p \\
                                        \end{smallmatrix}
                                       \right)
} M(u)\rightarrow 0$$
is an AR-sequence.
\end{itemize}
\end{prop}

\begin{thm}\cite{[BR]} \label{thm2.1}
Let $A$ be a string algebra. The Auslaner-Reiten sequences in $A\textup{-mod}$ are those described in Propositions \ref{prop1.1} and \ref{prop1.2}, together with those of the form
$$0\rightarrow M(w,s,\varphi)\rightarrow M(w,2s,\varphi')\rightarrow M(w,s,\varphi)\rightarrow 0$$
where $w\in\overline{\textup{Ba}}(A)$, $(K^s,\varphi)\in \mathcal{M}$  and $\varphi'$ is determined by the following AR-sequence in $K[T,T^{-1}]\textup{-mod}$,
$$0\rightarrow (K^s,\varphi)\rightarrow (K^{2s},\varphi')\rightarrow (K^s,\varphi)\rightarrow 0.$$
\end{thm}

\section{The Auslaner-Reiten quivers and $\tau$-locally free modules of  string algebras of type $\widetilde{C}_{n-1}$}

In this section, we introduce the notion of minimal string modules to study the
 AR-quiver of  string algebras $H$ of type $\widetilde{C}_{n-1}$. We will explicitly describe all the connected components of
 the AR-quiver and $\tau$-locally free $H$-modules.

\subsection{Minimal string modules} In this subsection,  $A$ can be any string algebra.
\begin{defn}\label{Def:minstmod}
A string $A$-module $M(w)$ is called {\em minimal}  if  in the AR-quiver of $A$, each irreducible map   $M(w)\rightarrow M(u)$ is injective and each irreducible map  $M(v)\rightarrow M(w)$ is surjective.
\end{defn}

We will see later in the next subsection that minimal string modules play an important role in determining the  connected components of the AR-quiver of the string algebra $H$ of type $\widetilde{C}_{n-1}$.
\begin{prop}\label{prop3.0}
There exists at least one minimal string module for each connected component of $\Gamma_A$ containing string modules.
\end{prop}

\begin{proof}
Assume that $\mathcal{T}$ is a connected  component containing a string module  $M(w)$.
{Theorem \ref{thm2.1} implies that all modules in $\mathcal{T}$ are string modules, as band modules are contained
in homogeneous tubes and any module in a homogeneous tube is a band module}.
If $M(w)$ is not minimal, then by definition there exists an irreducible surjection $f_1:M(w)\rightarrow M(w_1)$ or an irreducible injection $g_1: M(w_1)\rightarrow M(w)$ for some string $w_1$.
In either case $\textup{dim} M(w_1)<\textup{dim} M(w)$.
If $M(w_1)$ is not minimal, then
repeat the same procedure to find a string module with smaller dimension. This procedure will terminate eventually.
Then we obtain  a minimal string module in the component.
\end{proof}

For a string $x$, denote by $[x]$ the equivalence class of $x$.

\begin{defn}
Let $w$ be a string.

(1) The  cardinality $\textup{I}_\textup{{l}}(w)$ of the set $\{[u]|g:M(u)\rightarrow M(w)\ \textup{ is irreducible}\}$ is called {\em the left index} of $w$.

(2) The cardinality $\textup{I}_{\textup{r}}(w)$ of the set $\{[v]|f:M(w)\rightarrow M(v)\ \textup{ is irreducible}\}$ is called  {\em the right index} of $w$.

(3) The pair $\textup{I}(w)=(\textup{I}_{\textup{l}}(w),\textup{I}_{\textup{r}}(w))$ is called {\em the index} of $w$. We say $w$ is of {\em type} $(a,b)$ if $\textup{I}_{\textup{l}}(w)=a$ and $\textup{I}_{\textup{r}}(w)=b$. In this case, we also say that the string module $M(w)$ is of {\em type} $(a,b)$.
\end{defn}

Recall that we write $u\sim w$ if the two strings $u$ and $w$ are equivalent and
$u\not\sim w$ otherwise.

\begin{lem} \label{rem3.1}
Let $w$ be a string.  Then
\begin{itemize}
\item[(1)] $\textup{I}(w)\in\{(0,1),(1,0),(1,1),(0,2),(2,0),(1,2),(2,1),(2,2)\}$.

\item[(2)] $\textup{I}_{\textup{l}}(w)=0$ if and only if $M(w)$ is a simple projective module.
Dually, $\textup{I}_{\textup{r}}(w)=0$ if and only if $M(w)$ is a simple injective module.

\item[(3)] $\textup{I}_{\textup{l}}(w)=1$ if and only if $M(w)$ is either a projective module such that $\textup{rad}M(w)$ is indecomposable or  $w\sim \alpha_{-}$ for some arrow $\alpha\in Q_1$. Dually,
$\textup{I}_{\textup{r}}(w)=1$ if and only if $M(w)$ is either an injective module such that $M(w)/\textup{soc}M(w)$ is indecomposable or $w\sim$ $_{-}\beta$ for some arrow $\beta\in Q_1$.
\end{itemize}
\end{lem}
\begin{proof}
(1)  follows from Theorem \ref{thm2.1}, (2) and (3) follow from the general Auslander-Reiten theory \cite{[AR]} and Proposition \ref{prop1.1}.
\end{proof}

\begin{lem}\label{lem3.0}
An indecomposable projective $A$-module $P$ is minimal if and only if $P$ is simple. Dually, an indecomposable injective $A$-module $I$ is minimal if and only if $I$ is simple.
\end{lem}

\begin{proof}
We only prove the first assertion, the second one follows by duality. If $P$ is simple, then $P$ is minimal by definition.
Next assume that $P$ is not simple. Then  $\textup{rad}P\neq0$ and the embedding $\textup{rad}P\hookrightarrow P$ is a right almost split map. So the restriction of the embedding to any indecomposable summand is irreducible. This implies that
$P$ is not minimal. Therefore if $P$ is minimal, then it is simple.
\end{proof}

\pagebreak

\begin{lem}\label{lem3.2}
Let $w$ be a string.
\begin{itemize}
\item[(1)]
The string module $M(w)$ is minimal of type $(1,1)$  if and only if
 $w\sim \alpha_{-}$ for some  $\alpha\in Q_1$ and $w\sim{} _-\beta$ for some  $\beta\in Q_1$.

\item[(2)] The string module $M(w)$ is minimal of type $(1,2)$ if and only if $M(w)$ is not injective,   $w \sim \alpha_{-}$ for some $\alpha\in Q_1$ but  $w\not\sim{}_-\beta$ for any $\beta\in Q_1$, and $w$ is RDE and LIE.

\item[(3)] The string module $M(w)$ is minimal of type $(2,1)$ if and only if $M(w)$ is not projective,
 $w\sim{}_-\alpha$ for some $\alpha\in Q_1$ but $w\not\sim \beta_{-}$ for any $\beta\in Q_1$, and $w$ is RIE and LDE.

\item[(4)] The string module $M(w)$ is minimal of type $(2,2)$ if and only if $M(w)$ is neither projective nor injective,
 $w\not\sim {}_-\alpha$ and $w\not\sim \alpha_{-}$ for any $\alpha\in Q_1$,
 and  $w$  is RDE, RIE, LDE and LIE.
\end{itemize}
\end{lem}

\begin{proof}
(1) follows from Lemma \ref{rem3.1} (3) and Lemma \ref{lem3.0}.
Now we prove (2).  If $M(w)$ is minimal of type $(1,2)$, then $M(w)$ is neither projective nor injective by Lemma \ref{lem3.0} and Lemma \ref{rem3.1} (2). Since $\textup{I}_{\textup{l}}(w)=1$ and $\textup{I}_{\textup{r}}(w)=2$, we have $w\sim \alpha_{-}$ for some arrow $\alpha\in Q_1$ but $w\not\sim  {}_{-}\beta$ for any $\beta\in Q_1$ by Lemma \ref{rem3.1} (3). Note that $w$ is minimal, $w$ is RDE and LIE by Proposition \ref{prop1.2}. The converse is a direct consequence of Proposition \ref{prop1.1} and Proposition \ref{prop1.2} (1).  (3) and (4) can be similarly proved, we skip the details.
\end{proof}

\subsection{String algebras of type $\widetilde{C}_{n-1}$}

{\it For the remaining part of this paper we assume that $H=H(C, \Omega, D)$ with
$C$ the Cartan matrix of type $\widetilde{C}_{n-1}$ and $D$ the minimal symmetrizer
$\textup{diag}(2, 1, \dots, 1, 2)$, unless otherwise stated.} That is, $H$ is the quotient path algebra
$KQ/I$, { where $Q$ is a quiver of type $A_n$ when the loops at 1 and $n$ are removed},
$$ \xymatrix{
\varepsilon_1 \circlearrowleft  1 \ar@{-}[r] & 2\ar@{-}[r]&  \cdots \ar@{-}[r]  &n\circlearrowleft \varepsilon_n\\
}$$
and $I=<\varepsilon_1^2,\varepsilon_n^2>$.
Then $H$ is a string algebra. Moreover it is a gentle algebra.  
We say $H$ is a {\em string algebra  of type} $\widetilde{C}_{n-1}$. In this subsection, we will first describe minimal string $H$-modules in a more concrete way, using Lemma \ref{lem3.2} and Proposition \ref{prop1.2}, and then construct explicitly connected components of the AR-quiver of $H$.

Let $Q^0$ be the quiver obtained from $Q$ by deleting the two loops $\varepsilon_1$ and $\varepsilon_n$. Then $Q^0$ is a quiver of type ${A}_n$. We will describe connected components of the AR-quiver of $H$.  The orientations of the arrows in $Q^0$ incident at 1 and $n$ are particularly relevant to the description of the component containing the indecomposable projective $A$-modules and the proof. There are four possibilities,  (Fig.~1), (Fig.~2) and their opposite quivers. By duality, we only need to consider the two cases, (Fig.~1) and (Fig.~2), where the difference is that both $1$ and $n$ are sources in $Q^0$ in (Fig.~2), but only 1 is a source in (Fig.~1).


 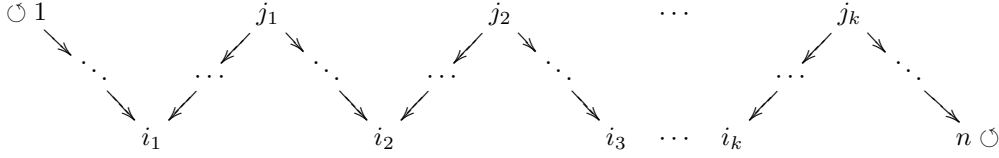
\begin{figure}
$$ \xymatrixcolsep{0.5pc}\xymatrixrowsep{0.5pc}\xymatrix{
\circlearrowleft1\ar[rd] &  &  &  &  j_1\ar[rd]\ar[ld] &  &  &  & j_2 \ar[rd]\ar[ld] &  & &\cdots & & & j_k \ar[rd]\ar[ld] & &  & \\
&\ddots\ar[rd]&&\cdots\ar[ld]&&\ddots\ar[rd]&&\cdots\ar[ld]&&\ddots\ar[rd]&&&&\cdots\ar[ld]&&\ddots\ar[rd]&\\
&& i_1 &&&& i_2 &&&& i_3&\cdots& i_k &&&& n\circlearrowleft\\
}$$
\caption{$Q$ is of type (3-1)}
\end{figure}

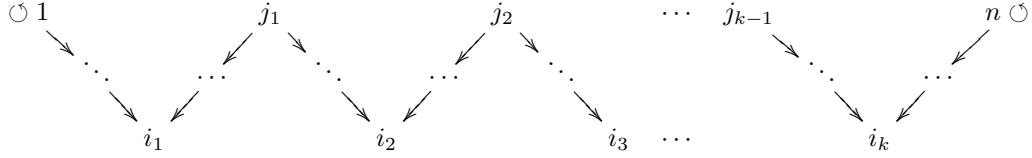
\begin{figure}
$$ \xymatrixcolsep{0.5pc}\xymatrixrowsep{0.5pc}\xymatrix{
\circlearrowleft1\ar[rd] &  &  &  &  j_1\ar[rd]\ar[ld] &  &  &  & j_2 \ar[rd]\ar[ld] &  & & \cdots & j_{k-1}\ar[rd] & & & &  n\circlearrowleft\ar[ld]\\
&\ddots\ar[rd]&&\cdots\ar[ld]&&\ddots\ar[rd]&&\cdots\ar[ld]&&\ddots\ar[rd]&&&&\ddots\ar[rd]&&\cdots\ar[ld]\\
&& i_1 &&&& i_2 &&&& i_3&\cdots& &&i_k&\\
}$$
\caption{ $Q$ is of type (3-2)}
\end{figure}


By the definition of $Q$, there is at most one arrow between two vertices $i, j$ and we denote the arrow  by $\alpha_{ji}$ if there is one from $i$ to $j$.
As the way the vertices are labelled,   $|i-j|=1$ if there is an arrow between two distinct
vertices $i$ and $j$.
{Recall that a vertex $i$ is {\em admissible} if it is a sink or a source.}

\begin{prop}\label{prop2.1} 
The following are the minimal string modules of $H$ up to isomorphism.
\begin{itemize}
\item[(1)] Type (0,2): $S_{i}$ with $i$ a sink  in $Q$.

\item[(2)] Type (2,0): $S_{i}$ with $i$ a source  in $Q$.

\item[(3)] Type (1,1): 
{ $M(\alpha_{-})$, where $\alpha$ is an arrow in $ Q^0$}.


\item[(4)] Type (1,2):
 $M((\varepsilon_1)_{-})$ and $M((\varepsilon_n)_{-})$.

\item[(5)] Type (2,1):  $M({}_{-}(\varepsilon_1))$ and $M({}_{-}(\varepsilon_n))$.

\item[(6)] Type (2,2): $M(w)$, where $w=c_1\dots c_m$ is a string with $\{s(w),t(w)\}\subseteq\{1,n\}$ and neither $c_1$ nor $c_m$ is a loop or the formal inverse of a loop. For instance, when the starting vertex and the ending vertex of $w$ are $n$ and 1 respectively, then  $w$ is a walk of the form
$$\xymatrix{1\ar@{-}[r]& 2\ar@{-}[r]&\cdots\ar@{-}[r]& n-1\ar@{-}[r] & n.}$$
\end{itemize}
\end{prop}

\begin{proof}
Observe that the only simple projective or injective modules are those corresponding to sinks or sources in $Q$,  which  can only occur in the middle of the quiver. So (1) and (2) are true.
(6) follows from Lemma \ref{lem3.2} (4).
 To prove (3), (4) and (5), we compute
 $M(\alpha_{-})$ and $M(_{-}\alpha)$ for all $\alpha\in Q_1$.

Case I:  $Q$ is of type (3-1) as in (Fig.~1). Set $j_0=1$ and $i_{k+1}=n$. Then we have the following:

\vspace{1mm}

\begin{itemize}
\item[(a)]
$M((\alpha_{p+1,p})_{-})=M({}_{-}(\alpha_{p,p-1}))=S_p\ \text{for} \ j_r<p<i_{r+1}\ \text{and}\ 0\leq r\leq k$.

\item[(b)] $M((\alpha_{q-1,q})_{-})=M({}_{-}(\alpha_{q,q+1}))=S_q\ \text{for} \ i_r<q<j_{r}\ \text{and}\ 1\leq r\leq k$.

\item[(c)]$M((\alpha_{21})_{-})=M({}_{-}(\alpha_{i_1,i_1+1}))$.

\item[(d)] $M((\alpha_{j_r+1,j_r})_{-})=M({}_{-}(\alpha_{i_r,i_r-1}))\ \text{for}\ 1\leq r\leq k$.

\item[(e)] $M((\alpha_{j_r-1,j_r})_{-})=M({}_{-}(\alpha_{i_{r+1},i_{r+1}+1}))\ \text{for}\ 1\leq r\leq k-1$.

\item[(f)] $M((\alpha_{j_k-1,j_k})_{-})=M({}_{-}(\alpha_{n,n-1}))$.

\item[(g)] $M(_{-}(\varepsilon_1))=S_1$ and
$M((\varepsilon_1)_{-})=\left\{
\begin{tabular}{ll} $M(\alpha_{21}^{-1}\cdots\alpha_{i_1,i_1-1}^{-1})$ & if $i_1\not = n$,\\ \\
$M(\alpha_{21}^{-1}\cdots\alpha_{i_1,i_1-1}^{-1} \varepsilon_n^{-1})$ & if $i_1 = n$.
\end{tabular}
\right. $

\vspace{1mm}

\item[(h)] $M((\varepsilon_n)_{-})=S_n$ and
$M(_{-}(\varepsilon_n))=\left\{
\begin{tabular}{ll} $M(\alpha_{j_k+1,j_k}^{-1}\cdots\alpha_{n,n-1}^{-1})$ & if $j_k\not = 1$,\\ \\
$M(\varepsilon_1^{-1}\alpha_{j_k+1,j_k}^{-1}\cdots\alpha_{n,n-1}^{-1})$ & if $j_k = 1$.
\end{tabular}
\right. $

\end{itemize}

By Lemma \ref{lem3.2} (1), cases (a)-(f) provide all the possible minimal string modules of type
$(1, 1)$. So  (3) holds.

By (a)-(h),  $(\varepsilon_1)_{-}$ and $(\varepsilon_n)_{-}$ are both RDE and LIE, but not of the form $_{-}\alpha$ for any $\alpha\in Q_1$. Since $M((\varepsilon_1)_{-})$ and $M((\varepsilon_n)_{-})=S_n$  are not injective, they are minimal string modules of type $(1, 2)$ by Lemma \ref{lem3.2} (2).
Moreover from the computation list, they are the only two such modules. So  (4) is true for $Q$ as in
(Fig. ~1).  Similarly, $M({}_{-}(\varepsilon_1))=S_1$ and $M({}_{-}(\varepsilon_n))$ are the only two minimal string modules of type
$(2, 1)$. So (5) is true for $Q$ as in (Fig.~1).

Case II:   $Q$ is of type (3-2) as in (Fig. 2). Set $j_0=1$ and $j_{k}=n$.
The difference between the two types of $Q$ is that $n$ is a sink in $Q$ of type (3-1), while it is a source in $Q$ of type (3-2). Similar computation shows that (3) and the claim for
$M((\varepsilon_1)_{-})$ and
$M(_{-}(\varepsilon_1))$ are true.
{ The remaining string  that has the right index equal to 1 is $_{-}(\varepsilon_n)=1_{(n,i)}$, which is RIE and LDE, but not of the form $\alpha_{-}$ for any $\alpha\in Q_1$ and $M(_{-}(\varepsilon_n))=S_n$ is not projective; and the remaining string that has left index equal to 1 is
$(\varepsilon_n)_{-}=\alpha_{n-1,n}^{-1}\cdots\alpha_{i_k,i_k+1}^{-1}$, which is RDE and LIE, but not of the form $_{-}\beta$ for any $\beta\in Q_1$ and $M((\varepsilon_n)_{-})$ is not injective. Note that $i_k$ is a sink and $1$ is a source, so $i_k\not=1$.}


So $M((\varepsilon_1)_{-})$ and $M((\varepsilon_n)_{-})$ are  the only two minimal string modules of type $(1, 2)$, and $M(_{-}(\varepsilon_1))$ and  $M(_{-}(\varepsilon_n))$ are the only two minimal string modules of type $(2, 1)$.
Therefore (4) and (5) hold. This completes the proof.
\end{proof}

Denote the Auslander-Reiten translation for $H$ by $\tau$.

\begin{prop}\label{prop3.3}
There are $n-1$  minimal string modules of type (1,1) (up to isomorphism) and they form the $\tau$-orbit at the bottom of
a tube of rank $n-1$. In particular,
for an arrow $\alpha: i\rightarrow j$ in $Q^0$, we have the following, depending on the properties of $i$ and
$j$ in $Q^0$.
\begin{itemize}
\item[(1)]  Both vertices $i$ and $ j$ are non-admissible. Then
\[  \tau S_i=S_j.\]

\item[(2)] The vertex $i$ is non-admissible and $j$ is a sink. Then
\[  \tau S_i=M(w_1),\]
where $w_1$ is the direct path of maximal length terminating at $j$ and satisfying that
$w_1^{-1}\alpha$ is a string. Note that by the definition of $Q$, such a path $w$ uniquely exists.

\item[(3)] The vertex $i$ is a source and $j$ is non-admissible. Then
\[  \tau^{-1} S_j=M(w_2),\]
where $w_2$ is the direct path of maximal length starting from $i$ and statisfying that
$\alpha w_2^{-1}$ is a string. Again, such a path $w$ uniquely exists.

\item[(4)] The vertex $i$ is a source and $j$ is a sink. Then
\[  \tau (M(w_2))=M(w_1),\]
where  $w_1$ and $w_2$ are paths satisfying the conditions on $w_1$ in (2) and the conditions
on $w_2$ in (3), respectively.
\end{itemize}
\end{prop}

\begin{proof} {
First, the claims in (1)-(4) are true, since the two modules in each case are $M(\alpha_{-})$ and
$M({}_{-}\alpha)$, respectively, for the arrow $\alpha$ in $Q^0$.  By Proposition \ref{prop2.1}, there are
exactly $n-1$ minimal string modules of type $(1, 1)$, one for each arrow in $Q^0$.
Next we prove that the $n-1$  minimal string modules form a $\tau$-orbit.

We connect two copies of $Q^0$ by $\varepsilon_1$ and $\varepsilon_n$, where $\varepsilon_1$ goes
from vertex 1 in the first copy to the 1 in the second copy and $\varepsilon_n$ goes
from vertex $n$ in the second copy to the $n$ in the first copy.
Denote the new quiver by $\tilde{Q}$ (see Example \ref{Ex:doublequiver} for an illustration). Observe that each arrow in $Q^0$ appears twice in $\tilde{Q}$, but in opposite directions, one goes anti-clockwise and the other one goes clockwise. So there are exactly
$n-1$ anti-clockwise arrows in $\tilde{Q}\backslash \{\varepsilon_1, ~\varepsilon_n\}$ and each arrow in $Q^0$ appears exactly once among the $n-1$ anti-clockwise arrows.

The computation (a) - (f) in the proof of Proposition \ref{prop2.1} can be interpreted as follows.
For any anti-clockwise  arrow $\gamma$  in $\tilde{Q}\backslash \{\varepsilon_1, \varepsilon_n\}$,
\[
M({}_{-}\gamma)=M(\beta_{-}),
\]
where $\beta$ is the next anti-clockwise arrow in $\tilde{Q}\backslash \{\varepsilon_1, ~\varepsilon_n\}$ after $\gamma$ when one walks along $\tilde{Q}$ anti-clockwise.
So
\[
\tau(M(\gamma_{-}))=M({}_{-}\gamma)=M(\beta_{-}).
\]
Continuing in this fashion, the $\tau$-orbit reaches all the $n-1$ minimal string modules of type $(1, 1)$ and stays
within these modules. Therefore the $n-1$ minimal string modules of type $(1, 1)$
form a $\tau$-orbit, and the $\tau$-orbit is at the bottom of the tube, because  these string modules
 are all of type $(1, 1)$. This completes the proof.}
\end{proof}

\begin{example}\label{Ex:doublequiver}
(1) {  Let $Q$ be the quiver: $\varepsilon_1 \circlearrowleft  1\xrightarrow{\alpha_{21}} 2 \xrightarrow{\alpha_{32}} 3 \xrightarrow{\alpha_{43}} 4\circlearrowleft \varepsilon_4$.
The quiver $\tilde{Q}$ constructed in the proof of Theorem \ref{prop3.3} is as follows, where the first copy of $Q$ is at the bottom,
$$
\xymatrix{
 1\ar[r] &2 \ar[r] & 3 \ar[r] &4\ar[d]^{\varepsilon_4} \\
1\ar[r] \ar[u]^{\varepsilon_1}&2 \ar[r] & 3 \ar[r] & 4
}
$$
and the modules at the bottom of the tube of rank 3 are
\[
\xymatrix{ M(\varepsilon_4 \alpha_{43} \alpha_{32}\alpha_{21}\varepsilon_1)
& S_3 &S_2 & M(\varepsilon_4 \alpha_{43} \alpha_{32}\alpha_{21}\varepsilon_1). }
\]
In the same order, these modules are
\[
\xymatrix{ M((\alpha_{21})_{-})
& M((\alpha_{43})_{-}) &M((\alpha_{32})_{-}) & M((\alpha_{21})_{-}). }
\]

\vspace{1mm}

(2)  Let $Q$ be the quiver: $\varepsilon_1 \circlearrowleft  1\xrightarrow{\alpha_{21}} 2 \xrightarrow{\alpha_{32}} 3 \xleftarrow{\alpha_{34}} 4\xrightarrow{\alpha_{54}} 5\circlearrowleft \varepsilon_5$.
Then the quiver $\tilde{Q}$ is
$$
\xymatrix{
 1\ar[r] &2 \ar[r] & 3  &4\ar[l]\ar[r]& 5\ar[d]^{\varepsilon_5}\\
1\ar[r] \ar[u]^{\varepsilon_1}&2 \ar[r] & 3 & 4\ar[l] \ar[r]& 5
}
$$
and the modules at the bottom of the tube of rank 4 are:
\[
\xymatrix{ M( \alpha_{32}\alpha_{21}\varepsilon_1)  & M(\varepsilon_5\alpha_{54})
& M(\alpha_{34})  &S_2 & M( \alpha_{32}\alpha_{21}\varepsilon_1). }
\]
In the same order, these modules are
\[
\xymatrix{ M((\alpha_{21})_{-}) & M((\alpha_{34})_{-})
& M((\alpha_{54})_{-}) &M((\alpha_{32})_{-}) & M((\alpha_{21})_{-}). }
\]
Each arrow in $Q^0$ appears exactly once.

\vspace{1mm}

(3) Let $Q$ be the quiver: $\varepsilon_1 \circlearrowleft  1\xrightarrow{\alpha_{21}} 2 \xrightarrow{\alpha_{32}} 3 \xleftarrow{\alpha_{34}} 4\circlearrowleft \varepsilon_4$.
Then the quiver $\tilde{Q}$ is
$$
\xymatrix{
 1\ar[r] &2 \ar[r] & 3  &4\ar[l]\ar[d]_{\varepsilon_4}
\\
1\ar[r] \ar[u]^{\varepsilon_1}&2 \ar[r] & 3 & 4\ar[l] }
$$
and the modules at the bottom of the tube of rank 4 are:
\[
\xymatrix{ M( \alpha_{32}\alpha_{21}\varepsilon_1)
& M(\alpha_{34} \varepsilon_4)  &S_2 & M( \alpha_{32}\alpha_{21}\varepsilon_1). }
\]}
In the same order, these modules are
\[
\xymatrix{ M((\alpha_{21})_{-})
& M((\alpha_{34})_{-}) &M((\alpha_{32})_{-}) & M((\alpha_{21})_{-}). }
\]

Note that in each case,  a module is the $\tau$-translation of the module on its immediate right.
\end{example}

\begin{lem}\label{lem:3.11}
Let $w$ be a string.
\begin{itemize}
\item[(1)] If $w$ is RDE, then so is $w_h$. Thus there exists the following infinite ray:
$$M(w\rightarrow w_{h^\bullet}):\ \ M(w)\rightarrow M(w_h)\rightarrow M(w_{h^{2}})\rightarrow\cdots$$
where $w_{h^i}=(w_{h^{i-1}})_h$ for $i\geq 2$.

\item[(2)] If $w$ is LIE, then so is $_hw$. Thus there exists the following infinite ray:
$$M(w\rightarrow {}_{h^\bullet}w):\ \ M(w)\rightarrow M(_hw)\rightarrow M(_{h^{2}}w)\rightarrow\cdots$$
where $_{h^i}w={}_h({}_{h^{i-1}}w)$ for $i\geq 2$.

\item[(3)] If $w$ is RIE, then so is $w_c$. Thus there exists the following infinite coray:
$$M(w_{c^\bullet}\rightarrow w): \ \ \cdots \rightarrow M(w_{c^{2}})\rightarrow M(w_c)\rightarrow M(w)$$
where $w_{c^i}=(w_{c^{i-1}})_c$  for $i\geq 2$.

\item[(4)] If $w$ is LDE, then so is $_cw$. Thus there exists the following infinite coray:
$$M({}_{c^\bullet}w\rightarrow w):\ \ \cdots\rightarrow M(_{c^{2}}w)\rightarrow M(_cw)\rightarrow M(w)$$
where $_{c^i}w={}_c({}_{c^{i-1}}w)$  for $i\geq 2$.
\end{itemize}
Moreover, the map at each step in the rays and corays is irreducible.
\end{lem}

\begin{proof}
First by the construction of the AR-sequences in Propositions \ref{prop1.1} and \ref{prop1.2},
the map at each step in the rays and corays is part of an AR-sequence and so it is irreducible.

(1) Suppose that $w$ is RDE. Then there exists $\alpha\in Q_1$ such that $w\alpha$ is a string. We have
\[w_h=\left\{
\begin{tabular}{ll} $w\alpha$ & if  $s(\alpha)$ is non-admissible in $Q^0$, \\ \\

$w\alpha w'$ & if $s(\alpha)$ is a source in $Q^0$
\end{tabular}
\right. \]
where $w\alpha w'$ is a string, $w'$ is an inverse string of maximal length. In particular,
$s(w')$ is a sink in $Q^0$.
In either case,  $w_h$ is again RDE. This proves (1). Similarly, (2) is true.

(3) Suppose that $w$ is RIE. Then there exists $\alpha \in Q_1$ such that $w\alpha^{-1}$ is a string.
Similar to (1), we have
\[w_c=\left\{
\begin{tabular}{ll} $w\alpha^{-1}$ & if  $t(\alpha)$ is non-admissible in $Q^0$, \\ \\

$w\alpha^{-1} w'$ & if $t(\alpha)$ is a sink in $Q^0$
\end{tabular}
\right. \]
where $w\alpha^{-1} w'$ is a string, $w'$ is a direct string of maximal length. In particular,
$s(w')$ is a source in $Q^0$.
In either case,  $w_c$ is again RIE. This proves (3). Similarly, (4) is true.
 \end{proof}

\begin{rem} Lemma \ref{lem:3.11} is not true in general. For instance, the linear quiver of type $A_n$ is a string algebra, but there is no infinite ray or coray.
\end{rem}

\begin{prop}\label{prop2.2}
There is a bijection between isomorphism classes of minimal string modules of type (2,2) and connected components of $\Gamma_H$ of type $\mathbb{Z}A_\infty^\infty$.
\end{prop}

\begin{proof}
Assume that $\mathcal{T}$ is a connected  component of $\Gamma_H$ of type $\mathbb{Z}A_\infty^\infty$. By Proposition \ref{prop3.0}, there  is a minimal  string module $M(w)$
occurring in $\mathcal{T}$ and so $M(w)$ is of type (2,2).

Conversely, assume that $M(w)$ is a minimal string module of type (2,2). Then the AR sequences containing $M(w)$ are as follows:
$$\xymatrixcolsep{0.8pc}\xymatrixrowsep{0.8pc}\xymatrix{
&& M(_hw_c) \ar[rd]^{-p}& & \\
& M(w_c)\ar[ru]^{i}\ar[rd]^{p} & & M(_hw)\ar[rd]^{-i}\\
M(_cw_c)\ar[ru]^{p}\ar[rd]^{-p}& & M(w)\ar[ru]^{i}\ar[rd]^{i} & & M(_hw_h)\\
&M(_cw)\ar[ru]^{p}\ar[rd]^{-i}&& M(w_h)\ar[ru]^{i}\\
& & M(_cw_h) \ar[ru]^{p}& &
}$$
 The connected  component $\mathcal{T}_w$ containing $M(w)$ is divided into four regions as follows by the  two rays
$M(w\rightarrow w_{h^\bullet}),M(w\rightarrow{}_{h^\bullet}w)$
and the two corays
$M(w_{c^\bullet}\rightarrow w), M({}_{c^\bullet}w\rightarrow w), $
$$\xymatrixcolsep{0.8pc}\xymatrixrowsep{0.8pc}\xymatrix{
\ddots\ar[rd]^{p}&&&&&& \vdots \\
&M(w_{c^2})\ar[rd]^{p} &  & (III) & & M(_{h^2}w)\ar[ru]^{i} & \\
&&M(w_c)\ar[rd]^{p} & & M(_hw)\ar[ru]^{i} & &  \\
&(IV)& &   M(w)\ar[ru]^{i}\ar[rd]^{i} &  &(I) & \\
&& M(_cw)\ar[ru]^{p} & & M(w_h)\ar[rd]^{i}&&\\
&M(_{c^2}w)\ar[ru]^{p} && (II) && M(w_{h^2})\ar[rd]^{i} & \\
\vdots\ar[ru]^{p}&&&&&&\ddots
}$$
By induction, we see that the AR-sequences in region (I), (II), (III) and (IV) are those in Proposition \ref{prop1.2} (1), (2), (3) and (4),   respectively.
In particular, $\textup{dim}M(w)<\textup{dim}M$ for any $M\in \mathcal{T}_w\backslash \{M(w)\}$ and
so minimal string modules of type (2, 2) (up to isomorphism) are in one to one correspondence with
components of type $\mathbb{Z}A_\infty^\infty$.
\end{proof}

\begin{rem}
Geiss \cite{[Gei]} describes modules of minimal dimension in a component of type
$\mathbb{Z}A_\infty^\infty$. Our  minimal string modules are defined differently
(see Definition \ref{Def:minstmod}) and
are defined for any component.
The proof of Proposition \ref{prop2.2} shows that for any minimal string $w$ of type $(2, 2)$,
$\textup{dim}M(w)<\textup{dim}M$ for any $M\in \mathcal{T}_w\backslash \{M(w)\}$. Therefore the minimal string
modules $M(w)$ of type $(2, 2)$ coincide with those described in
\cite[Proposition 3]{[Gei]}.
\end{rem}

{ Denote by $\mathcal{S}$ a complete set of representatives of simple $K[T, T^{-1}]$-modules.}

\begin{thm}\label{thm3.1} 
The AR-quiver $\Gamma_H$ of $H$ consists of the following:

\begin{itemize}
\item[(1)] one component $\mathcal{T}_{PI}$ containing all the indecomposable preprojective modules and all the indecomposable preinjective modules (up to isomorphism).

\item[(2)] { one tube of rank $n-1$, where the sum of the dimension vectors of the indecomposable modules at the bottom of the tube is $d=(d_i)$ with $d_i=2$ for all $i$.
(Note: we will see later if we take the sum of the rank vectors instead, then the sum is exactly $\delta$, the minimal positive imaginary root of type $C$). }

\item[(3)] homogeneous tubes $\mathcal{H}_{w,S}$, where $w\in\overline{\mathrm{Ba}}(H)$  and
$S\in \mathcal{S}$ is a simple module of the Laurent polynomial ring $K[T,T^{-1}]$.

\item[(4)]components $\mathcal{T}_\lambda$ of type $\mathbb{Z}A_\infty^\infty$,  where $\lambda$ runs through all the isomorphism classes of minimal string modules of type (2,2).
\end{itemize}
\end{thm}

\newcommand{\soc}{\operatorname{soc}}
\newcommand{\rad}{\operatorname{rad}}

\begin{proof}
Recall that for any indecomposable projective module $P$ and any indecomposable  injective module $I$, the natural embedding $\textup{rad} P \rightarrow P$ and the natural projection
$I\rightarrow I/\textup{soc}I$ are almost  split maps, see \cite{[AR]} for more details. We compute radicals of the indecomposable projective modules and quotients by socles of the indecomposable injective modules.
{By duality, we may
assume that $Q$ is of type (3-1) or (3-2) as in (Fig.~1) or (Fig.~2).}
\begin{itemize}
\item[(1)] $\rad P_1=P_2\oplus M((\varepsilon_1)_{-})$.
\item[(2)] For $1<i<n$:
$\rad P_i=\left\{ \begin{tabular}{ll} $P_j$ & if $i$ is not a source and $i\rightarrow j$,\\
$P_{i-1}\oplus P_{i+1}$  & if $i$ is a source. \end{tabular}\right.$
\item[(3)] $\rad P_{n}=\left\{ \begin{tabular}{ll} $S_n=M((\varepsilon_n)_{-})$
& if  $Q$ is of type (3-1), \\
$P_{n-1}\oplus M((\varepsilon_n)_{-})$ & if $Q$ is of type (3-2).
   \end{tabular}\right. $
\item[(4)] $I_1/\soc I_1=S_1=M(_{-}(\varepsilon_1))$.

\item[(5)] For $1<i<n$: $I_i/\soc I_i =\left\{ \begin{tabular}{ll} $I_j$ & if $i$ is not a sink and
$j\rightarrow i$, \\ $I_{i-1}\oplus I_{i+1}$ & if $i$ is a sink. \end{tabular}\right.$

\item[(6)] $I_n/\soc I_n=\left\{ \begin{tabular}{ll} $I_{n-1}\oplus M({}_{-}(\varepsilon_n))$  & if  $Q$ is of type (3-1), \\
$S_n=M({}_{-}(\varepsilon_n))$ & if $Q$ is of type (3-2). \end{tabular}\right.$
\end{itemize}
So in the AR-quiver, the indecomposable projective modules are in one slice, connected by irreducible maps and the same for the indecomposable injective modules.
Note that the orientation of the arrow between vertices $n-1$ and $n$ are different in the two quivers (Fig.~1) and (Fig.~2) and so strings $(\varepsilon_n)_{-}$ and ${}_{-}(\varepsilon_n)$
are different for the two quivers. However, in either case,
the AR-sequences
\[
0 \rightarrow M({}_{-}(\varepsilon_1))\rightarrow M({}_{-}(\varepsilon_1) \varepsilon_1 (\varepsilon_1)_{-})
\rightarrow M((\varepsilon_1)_{-})\rightarrow 0
\]
and
\[
0 \rightarrow M({}_{-}(\varepsilon_n))\rightarrow M({}_{-}(\varepsilon_n) \varepsilon_n (\varepsilon_n)_{-})
\rightarrow M((\varepsilon_n)_{-})\rightarrow 0
\]
connect the slice of injective modules and the slice of projective modules in the AR-quiver. In particular, the indecomposable projective and the indecomposable injective modules are in one component, denoted by
$\mathcal{T}_{PI}$, and this component contains all the
minimal string modules of type (0,2), (2,0), (1,2) and (2,1).  See (Fig.~3) and (Fig.~4),
 where
$M_1=M((\varepsilon_1)_-)$, $N_1=M({}_{-}(\varepsilon_1) \varepsilon_1 (\varepsilon_1)_{-})$,
$N_n=M({}_{-}(\varepsilon_n) \varepsilon_n (\varepsilon_n)_{-})$,
 $M_n=M({}_{-}(\varepsilon_n))$  and
 $M'_n=M((\varepsilon_n)_{-})$.

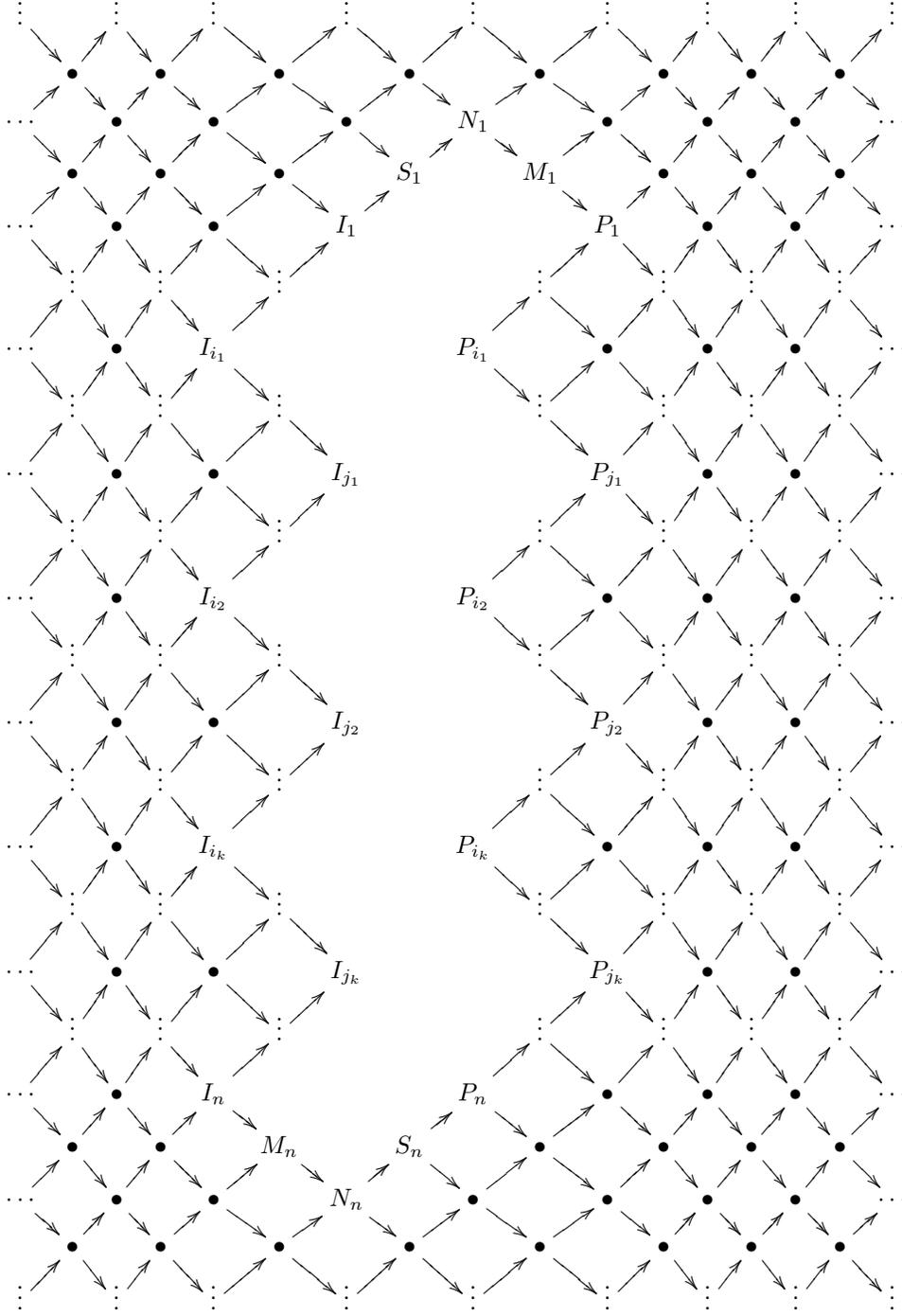
\begin{figure}
$$\xymatrixcolsep{0.55pc}\xymatrixrowsep{0.55pc}\xymatrix{
{\vdots}\ar[rd]&&{\vdots}\ar[rd]&&{\vdots}\ar[rd] && {\vdots}\ar[rd] &&{\vdots}\ar[rd] &&{\vdots}\ar[rd] &&{\vdots}\ar[rd] &&{\vdots}\ar[rd] &&{\vdots}\\
&\bullet\ar[ru]\ar[rd] && \bullet\ar[ru]\ar[rd] &&\bullet\ar[ru]\ar[rd]&&\bullet\ar[ru]\ar[rd]&&\bullet\ar[rd]\ar[ru]&&\bullet\ar[rd]\ar[ru]&&\bullet\ar[rd]\ar[ru]&&\bullet\ar[rd]\ar[ru] &&\\
{\cdots}\ar[ru]\ar[rd] &&\bullet \ar[ru]\ar[rd] &&\bullet\ar[ru]\ar[rd]&&\bullet\ar[ru]\ar[rd]&& N_1 \ar[rd]\ar[ru]&&\bullet\ar[rd]\ar[ru]&& \bullet\ar[rd]\ar[ru]&&\bullet\ar[rd]\ar[ru]&&{\cdots}\\
&\bullet\ar[ru]\ar[rd]&&\bullet\ar[ru]\ar[rd] &&\bullet\ar[ru]\ar[rd]&&S_1\ar[ru]&&M_1\ar[rd]\ar[ru]&&\bullet\ar[rd]\ar[ru]&&\bullet\ar[rd]\ar[ru]&&\bullet\ar[rd]\ar[ru]&& \\
{\cdots}\ar[ru]\ar[rd] &&\bullet \ar[ru]\ar[rd] &&\bullet\ar[ru]\ar[rd]&& I_1\ar[ru]&& && P_1\ar[rd]\ar[ru]&& \bullet\ar[rd]\ar[ru]&&\bullet\ar[rd]\ar[ru]&&{\cdots}\\
&\vdots\ar[ru]\ar[rd]&&\vdots\ar[ru]\ar[rd] &&\vdots\ar[ru]&&&&\vdots\ar[rd]\ar[ru]&&\vdots\ar[rd]\ar[ru]&&\vdots\ar[rd]\ar[ru]&&\vdots\ar[rd]\ar[ru]&& \\
{\cdots}\ar[ru]\ar[rd] &&\bullet \ar[ru]\ar[rd] &&I_{i_1}\ar[ru]\ar[rd]&&&&P_{i_1}\ar[ru]\ar[rd] &&\bullet\ar[rd]\ar[ru]&& \bullet\ar[rd]\ar[ru]&&\bullet\ar[rd]\ar[ru]&&{\cdots}\\
&\vdots\ar[rd]\ar[ru]&&\vdots\ar[ru]\ar[rd] &&\vdots\ar[rd]&&&&\vdots\ar[rd]\ar[ru]&&\vdots\ar[rd]\ar[ru]&&\vdots\ar[rd]\ar[ru]&&\vdots\ar[rd]\ar[ru]&& \\
{\cdots}\ar[ru]\ar[rd] &&\bullet \ar[ru]\ar[rd] &&\bullet\ar[ru]\ar[rd]&& I_{j_1}&& && P_{j_1}\ar[rd]\ar[ru]&& \bullet\ar[rd]\ar[ru]&&\bullet\ar[rd]\ar[ru]&&{\cdots}\\
&\vdots\ar[ru]\ar[rd]&&\vdots\ar[ru]\ar[rd] &&\vdots\ar[ru]&&&&\vdots\ar[rd]\ar[ru]&&\vdots\ar[rd]\ar[ru]&&\vdots\ar[rd]\ar[ru]&&\vdots\ar[rd]\ar[ru]&& \\
{\cdots}\ar[ru]\ar[rd] &&\bullet \ar[ru]\ar[rd] &&I_{i_2}\ar[ru]\ar[rd]&& && P_{i_2}\ar[rd]\ar[ru] && \bullet\ar[rd]\ar[ru]&& \bullet\ar[rd]\ar[ru]&&\bullet\ar[rd]\ar[ru]&&{\cdots}\\
&\vdots\ar[ru]\ar[rd]&&\vdots\ar[ru]\ar[rd] &&\vdots\ar[rd]&&&&\vdots\ar[rd]\ar[ru]&&\vdots\ar[rd]\ar[ru]&&\vdots\ar[rd]\ar[ru]&&\vdots\ar[rd]\ar[ru]&& \\
{\cdots}\ar[ru]\ar[rd] &&\bullet \ar[ru]\ar[rd] &&\bullet\ar[ru]\ar[rd]&& I_{j_2}&& && P_{j_2}\ar[rd]\ar[ru]&& \bullet\ar[rd]\ar[ru]&&\bullet\ar[rd]\ar[ru]&&{\cdots}\\
&\vdots\ar[ru]\ar[rd]&&\vdots\ar[ru]\ar[rd] &&\vdots\ar[ru]&&&&\vdots\ar[rd]\ar[ru]&&\vdots\ar[rd]\ar[ru]&&\vdots\ar[rd]\ar[ru]&&\vdots\ar[rd]\ar[ru]&& \\
{\cdots}\ar[ru]\ar[rd] &&\bullet \ar[ru]\ar[rd] &&I_{i_k}\ar[ru]\ar[rd]&& && P_{i_k}\ar[rd]\ar[ru] && \bullet\ar[rd]\ar[ru]&& \bullet\ar[rd]\ar[ru]&&\bullet\ar[rd]\ar[ru]&&{\cdots}\\
&\vdots\ar[ru]\ar[rd]&&\vdots\ar[ru]\ar[rd] &&\vdots\ar[rd]&&&&\vdots\ar[rd]\ar[ru]&&\vdots\ar[rd]\ar[ru]&&\vdots\ar[rd]\ar[ru]&&\vdots\ar[rd]\ar[ru]&& \\
{\cdots}\ar[ru]\ar[rd] &&\bullet \ar[ru]\ar[rd] &&\bullet\ar[ru]\ar[rd]&& I_{j_k}&& && P_{j_k}\ar[rd]\ar[ru]&& \bullet\ar[rd]\ar[ru]&&\bullet\ar[rd]\ar[ru]&&{\cdots}\\
&\vdots\ar[ru]\ar[rd]&&\vdots\ar[ru]\ar[rd] &&\vdots\ar[ru]&&&&\vdots\ar[rd]\ar[ru]&&\vdots\ar[rd]\ar[ru]&&\vdots\ar[rd]\ar[ru]&&\vdots\ar[rd]\ar[ru]&& \\
{\cdots}\ar[ru]\ar[rd] &&\bullet \ar[ru]\ar[rd] &&I_n\ar[ru]\ar[rd]&&&&P_n\ar[rd]\ar[ru]&&\bullet\ar[rd]\ar[ru]&& \bullet\ar[rd]\ar[ru]&&\bullet\ar[rd]\ar[ru]&&{\cdots}\\
&\bullet\ar[ru]\ar[rd]&&\bullet\ar[ru]\ar[rd] &&M_n\ar[rd]&&S_n\ar[rd]\ar[ru]&&\bullet\ar[rd]\ar[ru]&&\bullet\ar[rd]\ar[ru]&&\bullet\ar[rd]\ar[ru]&&\bullet\ar[rd]\ar[ru]&& \\
{\cdots}\ar[ru]\ar[rd] &&\bullet \ar[ru]\ar[rd] &&\bullet\ar[ru]\ar[rd]&&N_n\ar[ru]\ar[rd]&&\bullet\ar[rd]\ar[ru]&&\bullet\ar[rd]\ar[ru]&& \bullet\ar[rd]\ar[ru]&&\bullet\ar[rd]\ar[ru]&&{\cdots}\\
&\bullet\ar[ru]\ar[rd] &&\bullet \ar[ru]\ar[rd] &&\bullet\ar[ru]\ar[rd]&&\bullet\ar[ru]\ar[rd]&&\bullet\ar[rd]\ar[ru]&&\bullet\ar[rd]\ar[ru]&& \bullet\ar[rd]\ar[ru]&&\bullet\ar[rd]\ar[ru]&&\\
{\vdots}\ar[ru] && \vdots\ar[ru]&&\vdots\ar[ru]&&\vdots\ar[ru]&&\vdots\ar[ru]&&\vdots\ar[ru]&&\vdots\ar[ru]&&\vdots\ar[ru] &&{\vdots}\\
}$$
\caption{The connected component $\mathcal{T}_{PI}$ for $Q$ of type (3-1)}
\end{figure}

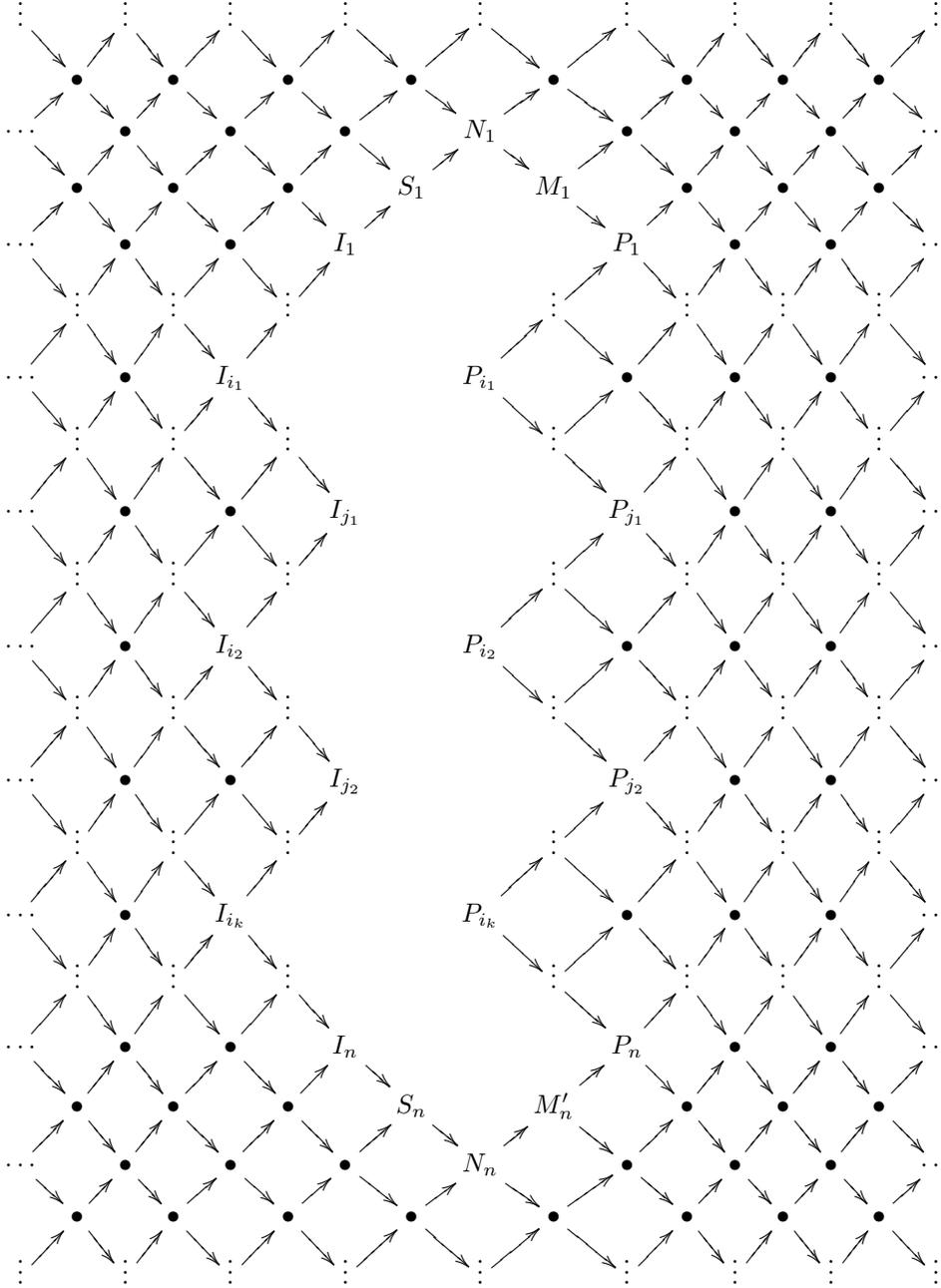
\begin{figure}
$$\xymatrixcolsep{0.6pc}\xymatrixrowsep{0.6pc}\xymatrix{
{\vdots}\ar[rd]&&{\vdots}\ar[rd]&&{\vdots}\ar[rd] && {\vdots}\ar[rd] &&{\vdots}\ar[rd] &&{\vdots}\ar[rd] &&{\vdots}\ar[rd] &&{\vdots}\ar[rd] &&{\vdots}\\
&\bullet\ar[ru]\ar[rd] && \bullet\ar[ru]\ar[rd] &&\bullet\ar[ru]\ar[rd]&&\bullet\ar[ru]\ar[rd]&&\bullet\ar[rd]\ar[ru]&&\bullet\ar[rd]\ar[ru]&&\bullet\ar[rd]\ar[ru]&&\bullet\ar[rd]\ar[ru] &&\\
{\cdots}\ar[ru]\ar[rd] &&\bullet \ar[ru]\ar[rd] &&\bullet\ar[ru]\ar[rd]&&\bullet\ar[ru]\ar[rd]&& N_1 \ar[rd]\ar[ru]&&\bullet\ar[rd]\ar[ru]&& \bullet\ar[rd]\ar[ru]&&\bullet\ar[rd]\ar[ru]&&{\cdots}\\
&\bullet\ar[ru]\ar[rd]&&\bullet\ar[ru]\ar[rd] &&\bullet\ar[ru]\ar[rd]&&S_1\ar[ru]&&M_1\ar[rd]\ar[ru]&&\bullet\ar[rd]\ar[ru]&&\bullet\ar[rd]\ar[ru]&&\bullet\ar[rd]\ar[ru]&& \\
{\cdots}\ar[ru]\ar[rd] &&\bullet \ar[ru]\ar[rd] &&\bullet\ar[ru]\ar[rd]&& I_1\ar[ru]&& && P_1\ar[rd]\ar[ru]&& \bullet\ar[rd]\ar[ru]&&\bullet\ar[rd]\ar[ru]&&{\cdots}\\
&\vdots\ar[ru]\ar[rd]&&\vdots\ar[ru]\ar[rd] &&\vdots\ar[ru]&&&&\vdots\ar[rd]\ar[ru]&&\vdots\ar[rd]\ar[ru]&&\vdots\ar[rd]\ar[ru]&&\vdots\ar[rd]\ar[ru]&& \\
{\cdots}\ar[ru]\ar[rd] &&\bullet \ar[ru]\ar[rd] &&I_{i_1}\ar[ru]\ar[rd]&&&&P_{i_1}\ar[ru]\ar[rd] &&\bullet\ar[rd]\ar[ru]&& \bullet\ar[rd]\ar[ru]&&\bullet\ar[rd]\ar[ru]&&{\cdots}\\
&\vdots\ar[rd]\ar[ru]&&\vdots\ar[ru]\ar[rd] &&\vdots\ar[rd]&&&&\vdots\ar[rd]\ar[ru]&&\vdots\ar[rd]\ar[ru]&&\vdots\ar[rd]\ar[ru]&&\vdots\ar[rd]\ar[ru]&& \\
{\cdots}\ar[ru]\ar[rd] &&\bullet \ar[ru]\ar[rd] &&\bullet\ar[ru]\ar[rd]&& I_{j_1}&& && P_{j_1}\ar[rd]\ar[ru]&& \bullet\ar[rd]\ar[ru]&&\bullet\ar[rd]\ar[ru]&&{\cdots}\\
&\vdots\ar[ru]\ar[rd]&&\vdots\ar[ru]\ar[rd] &&\vdots\ar[ru]&&&&\vdots\ar[rd]\ar[ru]&&\vdots\ar[rd]\ar[ru]&&\vdots\ar[rd]\ar[ru]&&\vdots\ar[rd]\ar[ru]&& \\
{\cdots}\ar[ru]\ar[rd] &&\bullet \ar[ru]\ar[rd] &&I_{i_2}\ar[ru]\ar[rd]&& && P_{i_2}\ar[rd]\ar[ru] && \bullet\ar[rd]\ar[ru]&& \bullet\ar[rd]\ar[ru]&&\bullet\ar[rd]\ar[ru]&&{\cdots}\\
&\vdots\ar[ru]\ar[rd]&&\vdots\ar[ru]\ar[rd] &&\vdots\ar[rd]&&&&\vdots\ar[rd]\ar[ru]&&\vdots\ar[rd]\ar[ru]&&\vdots\ar[rd]\ar[ru]&&\vdots\ar[rd]\ar[ru]&& \\
{\cdots}\ar[ru]\ar[rd] &&\bullet \ar[ru]\ar[rd] &&\bullet\ar[ru]\ar[rd]&& I_{j_2}&& && P_{j_2}\ar[rd]\ar[ru]&& \bullet\ar[rd]\ar[ru]&&\bullet\ar[rd]\ar[ru]&&{\cdots}\\
&\vdots\ar[ru]\ar[rd]&&\vdots\ar[ru]\ar[rd] &&\vdots\ar[ru]&&&&\vdots\ar[rd]\ar[ru]&&\vdots\ar[rd]\ar[ru]&&\vdots\ar[rd]\ar[ru]&&\vdots\ar[rd]\ar[ru]&& \\
{\cdots}\ar[ru]\ar[rd] &&\bullet \ar[ru]\ar[rd] &&I_{i_k}\ar[ru]\ar[rd]&& && P_{i_k}\ar[rd]\ar[ru] && \bullet\ar[rd]\ar[ru]&& \bullet\ar[rd]\ar[ru]&&\bullet\ar[rd]\ar[ru]&&{\cdots}\\
&\vdots\ar[ru]\ar[rd]&&\vdots\ar[ru]\ar[rd] &&\vdots\ar[rd]&&&&\vdots\ar[rd]\ar[ru]&&\vdots\ar[rd]\ar[ru]&&\vdots\ar[rd]\ar[ru]&&\vdots\ar[rd]\ar[ru]&& \\
{\cdots}\ar[ru]\ar[rd] &&\bullet \ar[ru]\ar[rd] &&\bullet\ar[ru]\ar[rd]&& I_{n}\ar[rd]&& && P_{n}\ar[rd]\ar[ru]&& \bullet\ar[rd]\ar[ru]&&\bullet\ar[rd]\ar[ru]&&{\cdots}\\
&\bullet\ar[ru]\ar[rd]&&\bullet\ar[ru]\ar[rd] &&\bullet\ar[rd]\ar[ru]&&S_n\ar[rd]&&M'_n\ar[rd]\ar[ru]&&\bullet\ar[rd]\ar[ru]&&\bullet\ar[rd]\ar[ru]&&\bullet\ar[rd]\ar[ru]&& \\
{\cdots}\ar[ru]\ar[rd] &&\bullet \ar[ru]\ar[rd] &&\bullet\ar[ru]\ar[rd]&&\bullet\ar[ru]\ar[rd]&&N_n\ar[rd]\ar[ru]&&\bullet\ar[rd]\ar[ru]&& \bullet\ar[rd]\ar[ru]&&\bullet\ar[rd]\ar[ru]&&{\cdots}\\
&\bullet\ar[ru]\ar[rd] &&\bullet \ar[ru]\ar[rd] &&\bullet\ar[ru]\ar[rd]&&\bullet\ar[ru]\ar[rd]&&\bullet\ar[rd]\ar[ru]&&\bullet\ar[rd]\ar[ru]&& \bullet\ar[rd]\ar[ru]&&\bullet\ar[rd]\ar[ru]&&\\
{\vdots}\ar[ru] && \vdots\ar[ru]&&\vdots\ar[ru]&&\vdots\ar[ru]&&\vdots\ar[ru]&&\vdots\ar[ru]&&\vdots\ar[ru]&&\vdots\ar[ru] &&{\vdots}\\
}$$
\caption{The connected component $\mathcal{T}_{PI}$ for $Q$ of type (3-2)}
\end{figure}

{By Proposition \ref{prop3.3},  all minimal string modules of type (1,1) form the $\tau$-orbit
at the bottom of the  tube of rank $n-1$.
By construction, each vertex appears exactly twice in the walks corresponding to the minimal strings. So  the sum of the dimension vectors of the minimal string modules has $2$ at all entries.

By Proposition \ref{prop2.2}, the isomorphism classes of minimal string modules of type (2,2) are in
one-to-one correspondence with connected components of type $\mathbb{Z}A_\infty^\infty$ . There are no other connected components containing string modules, following Proposition \ref{prop3.0}.}

Finally, by Theorem \ref{BRThm}, it remains to consider components consisting of band modules.
We know from Theorem \ref{thm2.1} that an indecomposable module is a band module if and only if it is contained in a  homogeneous tube. Each homogenous tube is uniquely determined by a band module $M(w,m,\varphi)$ for $w\in\overline{\textup{Ba}}(H)$ and $S=(K^m,\varphi)\in \mathcal{S}$ is a simple module over $K[T,T^{-1}]$.
So the connected components of the AR-quiver of $H$ are as described in the theorem.
\end{proof}

\subsection{$\tau$-locally free modules}
Let $e_1,e_2,\cdots,e_n$ be the idempotents in $H$ corresponding to the vertices of $Q$ and let $H_i=e_iHe_i$ for $1\leq i\leq n$. Then
$$H_i\cong \left\{
            \begin{array}{ll}
              K[\varepsilon_i]/(\varepsilon_i^2) & \text{if}\ i=1,n, \\
              K & \text{if}\ 2\leq i\leq n-1. \\
            \end{array}
          \right.
$$

\begin{defn}
A left $H$-module $M$ is called {\em locally free} if $M_i=e_iM$ is a free $H_i$-modules for each $i\in Q_0$.
An indecomposable locally free $H$-module $M$ is called {\em $\tau$-locally free}, if $\tau^k (M)$ is locally free for all $k\in\mathbb{Z}$.
\end{defn}

\begin{lem}\label{lem3.6}
The following are true.
\begin{itemize}
\item[(1)] If $M(w)$ is a minimal string module of type (1,2),  then the modules $M(w_{h^i})$ are not locally free, where
$i\geq 0$.

\item[(2)] If $M(w)$ is a minimal string module of type (2,1), then the modules $M({}_{c^i}w)$ are not locally free,
where   $i\geq0$.

\item[(3)] If $M(w)$ is a minimal string module of type (2,2), then $M(w)$ is not locally free. Moreover, none of the modules $M(w_{h^i})$, $M(_{h^i}w)$ $M(w_{c^i})$ and $M(_{c^i}w)$ $(i\geq 1)$
is locally free.
\end{itemize}
\end{lem}

\begin{proof}
(1) If $M(w)$ is a minimal string module of type (1,2), then $w=(\varepsilon_1)_{-}$ or $w=(\varepsilon_n)_{-}$,  by Proposition
\ref{prop2.1}. Without loss of generality, we assume that $w=(\varepsilon_n)_{-}$. If $n$ is a sink in $Q^0$, then
$w=1_{(n,t)}$ and all the other $w_{h^{\bullet}}$ have the form:
\[
\xymatrix{n& n-1\ar[l] \ar@{-}[r]& \bullet \ar@{-}[r]&\cdots}
\]
that is, it ends with the arrow $\alpha_{n-1, n}$. Therefore none of $M(w_{h^{\bullet}})$ is locally free.
Similarly, when $n$ is a source in $Q^0$, $w_{h^{\bullet}}$ have the form:
\[
\xymatrix{n\ar[r]& n-1 \ar@{-}[r]& \bullet \ar@{-}[r]&\cdots}
\]
and so none of $M(w_{h^{\bullet}})$ is locally free either.

Similarly,   (2) holds.

(3) If $M(w)$ is a minimal string module of type (2,2), then by Proposition \ref{prop2.1}, the string $w$ is one of the following form
$$\xymatrix{1\ar@{-}[r]& 2\ar@{-}[r]&\cdots\ar@{-}[r]& n-1\ar@{-}[r] & n},$$
$$\xymatrix{1\ar@{-}[r]& 2\ar@{-}[r]&\cdots\ar@{-}[r]& 2\ar@{-}[r] & 1},$$
$$\xymatrix{n\ar@{-}[r]& n-1\ar@{-}[r]&\cdots\ar@{-}[r]& n-1\ar@{-}[r] & n}.$$
So by similar arguments as in (1), $M(w)$ is not locally free and none of the modules
$M(w_{h^i})$, $M(_{h^i}w)$ (resp. $M(w_{c^i})$, $M(_{c^i}w)$)
 obtained by repeatedly adding hooks (resp.  cohooks) from either the right or the left (but not both)
 is locally free. This completes the proof.
\end{proof}

Let $w_0$ be the shortest walk with $s(w_0)=1$ and $t(w_0)=n$, consisting of all the arrows in $Q^0$.
We also denote the corresponding string starting from 1 and terminating at $n$ by $w_0$.
Following the definition of a band, we have following.

\begin{lem}\label{lem:band}
 Any band $w$ is equivalent to a band of the  standard form
 \begin{center}
 $w_0^{-1}\varepsilon_n^{\pm} w_0 \varepsilon_1^{\pm}... w_0\varepsilon_1^{\pm},$\end{center}
 whose starting vertex and terminating vertex are both $1$.
 In particular,  each time the walk of $w$ reaches vertices $1$ and $n$ in the middle of the walk (i.e. different from $s(w)$ and $t(w)$), it goes via the loops at these vertices.
\end{lem}

\begin{example} The following are all bands of the standard form:
\begin{center}
 $w_0^{-1}\varepsilon_n w_0 \varepsilon_1,  ~~ w_0^{-1}\varepsilon_n w_0 \varepsilon_1^{-1},  ~~
 w_0^{-1}\varepsilon_nw_0\varepsilon_1^{-1}w_0^{-1} \varepsilon_n w_0\varepsilon_1,$\end{center}
 where the last band is a composition of the first two.
\end{example}

\begin{thm}\label{thm3.2}
 Let 
 $M$ be an indecomposable $H$-module. Then $M$ is $\tau$-locally free if and only if one of the following is satisfied.
\begin{itemize}
\item[(1)] $M$ is preprojective.

\item[(2)] $M$ is preinjective.

\item[(3)] $M$ is a regular module occurring in any tube.
\end{itemize}
\end{thm}

\begin{proof}
Any preprojective module $\tau^{i}P_j$ and any preinjective module $\tau^s I_t$ are rigid, and so they are $\tau$-locally free by \cite[Proposition 11.6]{[GLS1]}.

Observe that the modules at the bottom of the tube of rank $n-1$ (see Proposition \ref{prop3.3}) are locally free and the other modules in the tube have a filtration by these modules and so are locally free as well. Therefore they are all $\tau$-locally free.

By Lemma \ref{lem:band}, an indecomposable band module is locally free and thus $\tau$-locally free, as such a module is in a homogeneous tube, i.e. a tube of rank 1.  Consequently, any indecomposable module in a homogeneous tube is $\tau$-locally free. Therefore the modules described in (1) - (3) are all $\tau$-locally free.

Next we show that there is no other $\tau$-locally free modules. First consider modules in any component $\mathcal{T}_w$ of type $\mathbb{Z}A_\infty^\infty$, where $w$ is the minimal string of type $(2, 2)$ that determines the component. By Lemma \ref{lem3.6},  modules in the rays and corays that divides $\mathcal{T}_w$  into 4 regions in the proof of Proposition \ref{prop2.2} are not locally free. Therefore any $\tau$-orbit in  $\mathcal{T}_w$  contains modules that are not locally free and so there is no $\tau$-locally free module in $\mathcal{T}_w$.

By Theorem \ref{thm3.1}, it remains to show that modules other than the preprojective and preinjective modules in the component $\mathcal{T}_{PI}$ are not $\tau$-locally free. Observe
 that the orbits of the other modules meet either the rays or the corays containing $S_1$ and
 $S_n$, respectively. As $S_1$ and $S_n$ are not locally free modules,
 modules in those rays/corays are not locally free by  Lemma \ref{lem3.6}. Therefore the modules
in $\mathcal{T}_{PI}$ that are neither preprojective nor preinjective are not $\tau$-locally free.
This completes the proof.
\end{proof}

\section{An application to the conjecture by Geiss-Lercler-Schr\"{o}er}

In this section, we  apply Theorem \ref{thm3.2} to prove Conjecture 1 in the case where the Cartan matrix is of type $\widetilde{C}_{n-1}$,
$$C=\left(
                \begin{array}{ccccccc}
                  2 & -1 &  &  & & \\
                  -2 & 2 & -1 &  &  &\\
                  & -1 & 2 & -1 & & \\
                 & & \ddots & \ddots & \ddots & \\
                &  & & -1 & 2 & -1 & \\
                 & & &  & -1 & 2 & -2 \\
                 & &  &  &  & -1 & 2 \\
                \end{array}
              \right)
$$
and the symmetrizer $D=\textup{diag}(2,1,1,\cdots,1,1,2)$.

\subsection{Roots and Coxeter transformations} In this subsection $C$ can be any symmetrizable Cartan $n\times n$ matrix of affine type and $D$ can be any symmetrizer of $C$.
Let $\alpha_1,\alpha_2,\cdots,\alpha_n$ be a list of  positive simple roots of type $C$.
For $1\leq i,j\leq n$, define
$$s_i(\alpha_j)=\alpha_j-c_{ij}\alpha_i.$$
This yields a reflection $s_i:\mathbb{Z}^{n}\rightarrow \mathbb{Z}^{n}$ on the root lattice $\mathbb{Z}^{n}=\sum_{i=1}^{n}\mathbb{Z}\alpha_i$, where $\alpha_i$ is identified with the $i$th standard basis vector of $\mathbb{Z}^{n}$.
The {\em Weyl group} $W$ is the subgroup of $\textup{Aut}(\mathbb{Z}^{n})$ generated by $s_1,s_2,\cdots,s_n$.
Denote by $$\Delta_{\textup{re}}=\cup_{i=1}^{n}W(\alpha_i)$$  the set of {\em real roots}, and by
 $$\Delta_{\textup{im}}=\mathbb{Z}\delta$$  the set of {\em imaginary roots},
 where $\delta$ is the unique minimal positive imaginary root determined by the Cartan matrix
 $C$.   For instance, in the case of
 type $\widetilde{C}_{n-1}$,  $$\delta=\alpha_1+2\Sigma_{i=2}^{n-1}\alpha_i+\alpha_n=(1,2,\cdots,2,1).$$
The set  of {\em roots} determined by $C$ is
 $$\Delta=\Delta_{\textup{re}}\cup\Delta_{\textup{im}}$$ and
 with the set of {\em positive roots}
$$\Delta^{+}=\Delta\cap \mathbb{N}^{n}=\Delta^+_{\textup{re}}\cup\Delta^+_{\textup{im}},$$
where $\Delta^+_{\textup{re}}=\Delta_{\textup{re}}\cap\mathbb{N}^{n}$ and $\Delta^+_{\textup{im}}=\Delta_{\textup{im}}\cap\mathbb{N}^{n}$.

An {\em orientation} of $C$ is a subset $\Omega\subset\{1,2,\cdots,n\}\times\{1,2,\cdots,n\}$ such that the following hold:

(1) $\{(i,j),(j,i)\}\cap\Omega\neq\emptyset$ if and only if $c_{ij}<0$;

(2) For each sequence $((i_1,i_2),(i_2,i_3),\cdots,(i_t,i_{t+1}))$ with $t\geq 1$ and $(i_s,i_{s+1})\in\Omega$ for all $1\leq s\leq t$, we have $i_1\neq i_{t+1}$.

Let $Q=Q(C,  \Omega)$ be the quiver with vertices $Q_0=\{1, \dots, n\}$ and arrows
\[Q_1=\{\alpha^g_{ji}: i\rightarrow j\mid  (j, i) \in \Omega \textup{ and } 1\leq g\leq \textup{gcd}\{|c_{ij}|, ~|c_{ji}|\}  \}\cup
\{\varepsilon_i: i\rightarrow i\mid i\in Q_0\}.\]
Let $Q^0=Q^0(C,\Omega)$ be the quiver obtained from $Q$ with the loops $\varepsilon_i$ removed.

For an orientation $\Omega$ of $C$ and an admissible vertex $i$ in $Q^0(C,\Omega)$, let
$$s_i(\Omega)=\{(r,s)\in\Omega\mid i\notin\{r,s\}\}\cup\{(s,r) \mid i\in\{r,s\},(r,s)\in\Omega\}.$$
Then $s_i(\Omega)$ is again an orientation of $C$.
A sequence $\mathbf{i}=(i_1,i_2,\cdots,i_n)$ is called a {\em $+$-admissible sequence} for $(C,\Omega)$ if $\{i_1,i_2,\cdots,i_n\}=\{1,2,\cdots,n\}$, $i_1$ is a sink in $Q^0(C,\Omega)$ and $i_k$ is a sink in $Q^0(C,s_{i_{k-1}}\cdots s_{i_1}(\Omega))$ for $2\leq k\leq n$. For such a sequence $\mathbf{i}$, define $$\beta_{\mathbf{i},k}=\left\{\begin{array}{ll}
                \alpha_{i_1} & \text{if}\ k=1, \\
                s_{i_1}s_{i_2}\cdots s_{i_{k-1}}(\alpha_{i_{k}}) & \text{if}\ 2\leq k\leq n.
              \end{array}
\right.$$
Similarly, define
$$\gamma_{\mathbf{i},k}=\left\{\begin{array}{ll}  \alpha_{i_n} & \text{if}\ k=n, \\
                s_{i_n}\cdots s_{i_{k+1}}(\alpha_{i_k}) & \text{if}\ 1\leq k\leq n-1. \\
                      \end{array}
\right.$$

Let $c_\mathbf{i}=s_{i_n}s_{i_{n-1}}\cdots s_{i_1}:\mathbb{Z}^{n}\rightarrow\mathbb{Z}^{n}$. Then $c_\mathbf{i}^{-1}=s_{i_1}s_{i_{2}}\cdots s_{i_n}:\mathbb{Z}^{n}\rightarrow\mathbb{Z}^{n}$. These are two Coxeter transformations  associated to $\mathbf{i}$.


For a $+$-admissible sequence $\mathbf{i}=(i_1,i_2,\cdots,i_n)$ for $(C,~\Omega)$, the rotated sequence
\[\mathbf{i}'=(i_2,i_3\cdots,i_n, i_1)\]
is also a $+$-admissible sequence for $(C,s_{i_1}(\Omega))$, and $c_{\mathbf{i'}}=s_{i_1}s_{i_n}\cdots s_{i_{3}} s_{i_2}$ and $c_{\mathbf{i'}}^{-1}=s_{i_2}s_{i_3}\cdots s_{i_{n}} s_{i_1}$ are the Coxeter transformations associated to $\mathbf{i'}$.

Similary, a {\em $-$-admissible sequence} can be defined using sources. In fact, the sequence
$\mathbf{i}=(i_1,i_2,\cdots,i_n)$ is $+$-admissible if and only if its reverse sequence
$\mathbf{i}^{-1}=(i_n,i_{n-1},\cdots,i_1)$ is $-$-admissible. We have,
$c_{\mathbf{i}}^{-1}=c_{\mathbf{i}^{-1}}$. Similar to $+$-admissible sequences, a rotated
sequence of a $-$-admissible sequence is also  $-$-admissible. We call  both
a  $+$-admissible sequence and a $-$-admissible sequence an {\em admissible sequence}.

For a $-$-admissible sequence $\mathbf{i}=(i_1, \dots, i_n)$, define $$\gamma_{\mathbf{i},k}=\left\{\begin{array}{ll}
                \alpha_{i_1} & \text{if}\ k=1, \\
                s_{i_1}s_{i_2}\cdots s_{i_{k-1}}(\alpha_{i_{k}}) & \text{if}\ 2\leq k\leq n,
              \end{array}
\right.$$
and
$$\beta_{\mathbf{i},k}=\left\{\begin{array}{ll}  \alpha_{i_n} & \text{if}\ k=n, \\
                s_{i_n}\cdots s_{i_{k+1}}(\alpha_{i_k}) & \text{if}\ 1\leq k\leq n-1. \\
                      \end{array}
\right.$$

\begin{lem}\label{lem4.0} Let $\mathbf{i}=(i_1,  \dots, i_n)$ be an admissible sequence and
$\mathbf{i}'=(i_2, \dots, i_n, i_1)$. Then the reflection $s_{i_1}$ induces a bijection between
$\{c_\mathbf{i}^{-r}(\beta_{\mathbf{i},k})|r\in\mathbb{Z}_{\geq 0},1\leq k\leq n\}\cup\{c_\mathbf{i}^{s}(\gamma_{\mathbf{i},k})|s\in\mathbb{Z}_{\geq 0},1\leq k\leq n\}\setminus\{\alpha_{i_1}\}$ and $\{c_{\mathbf{i}'}^{-r}(\beta_{\mathbf{i}',k})|r\in\mathbb{Z}_{\geq 0},1\leq k\leq n\}\cup\{c_{\mathbf{i}'}^{s}(\gamma_{\mathbf{i}',k})|s\in\mathbb{Z}_{\geq 0},1\leq k\leq n\}\setminus\{\alpha_{i_1}\}.$
\end{lem}

\begin{proof} First consider the case where $\mathbf{i}$ is $+$-admissible.
Note that $s_{i_1}c_{\mathbf{i}'}=c_{\mathbf{i}}s_{i_1}$, $s_{i_1}c^{-1}_{\mathbf{i}}=c^{-1}_{\mathbf{i}'}s_{i_1}$ and $\beta_{\mathbf{i},1}=\gamma_{\mathbf{i}',n}=\alpha_{i_1}$ by definition. The lemma follows from the following calculation
$$s_{i_1}(c_\mathbf{i}^{-r}(\beta_{\mathbf{i},k}))=\left\{\begin{array}{ll}
                                                     c_{\mathbf{i}'}^{-r}(\beta_{\mathbf{i}',k-1})& \text{if}\ 2\leq k\leq n, r\geq 0, \\
                                                      c_{\mathbf{i}'}^{-r+1}(\beta_{\mathbf{i}',n}) & \text{if}\ k=1, r>0
\end{array}\right.
$$ and
$$s_{i_1}(c_\mathbf{i}^{s}(\gamma_{\mathbf{i},k}))=\left\{\begin{array}{ll}
                                                     c_{\mathbf{i}'}^{s}(\gamma_{\mathbf{i}',k-1}) & \text{if}\ 2\leq k\leq n, \\
                                                      c_{\mathbf{i}'}^{s+1}(\gamma_{\mathbf{i}',n}) & \text{if}\ k=1
\end{array}\right.
$$ for each $s\geq 0$.

When $\mathbf{i}$ is $-$-admissible, the proof can be similarly done. We skip the details.
\end{proof}

\subsection{Geiss-Leclerc-Schr\"{o}er's Conjecture} In this subsection, we will prove Conjecture 1 for the case where
$C$ is of type $\widetilde{C}_{n-1}$ and the symmetrizer $D$ is minimal, that is, $D=\mathrm{diag}(2, 1, \dots, 2, 1)$.

For  a locally free $H$-module  $M$,   denote by $r_i$ the rank of the free $H_i$-module $M_i$, where
$i\in Q_0$.
We call $$\underline{\textup{rank}}(M):=(r_1,\cdots,r_n)$$ the {\em rank vector} of $M$.

Below we recall a few results from  \cite{[GLS1]},  which are important to prove the main result Theorem \ref{thm4.1} in this section.

\begin{lem}\cite[Proposition 11.5]{[GLS1]} \label{lem4.1}
Let $c=c_{\mathbf{i}}$ for some $+$-admissible sequence $\mathbf{i}=(i_1, \dots, i_n)$ and
$M$ be a $\tau$-locally free $H$-module. If $\tau^k(M)\neq 0$, then
\[\underline{\textup{rank}}(\tau^k(M))=c^k(\underline{\textup{rank}}(M)).\]
\end{lem}

\begin{lem}\cite[Lemmas 2.1, 3.2 and 3.3]{[GLS1]} \label{lem4.3}
Let $C$ be a symmetrizable Cartan matrix that is not of Dynkin type and let $\mathbf{i}=(i_1,i_2,\cdots,i_n)$ be an admissible sequence. Then
\[\underline{\textup{rank}}(\tau^{-r}(P_{i_k}))=c_\mathbf{i}^{-r}(\beta_{\mathbf{i},k})\] and
\[\underline{\textup{rank}}(\tau^{s}(I_{i_k}))=c_\mathbf{i}^s(\gamma_{\mathbf{i},k}),\] where
 $r,s\geq 0$ and $1\leq k\leq n$. Moreover these rank vectors  are pairwise distinct positive real roots. \end{lem}


Note that a representation of $Q=Q(C, \Omega)$ can be naturally viewed as a representation of a modulated graph
$\mathcal{M}(C, D)$ and vice versa. The representation categories of $Q$
and $\mathcal{M}(C, D)$ are equivalent. For a sink (resp. a source) in the modulated graph, one can define a reflection functor $F_i^{+}$ (resp. $F_i^{-}$) on the representations of the modulated graph, in a similar way as reflection functors defined for (simply-laced) quivers. When $i$
is admissible, we write the reflection functor by $F_i$ which should be interpreted as $F_i^+$ when
$i$ is a sink and $F_i^-$ otherwise.

\begin{lem}\cite[Proposition 9.4]{[GLS1]}\label{lem4.2}
Let $H=H(C,D,\Omega)$ and $H'=H(C,D,s_{i}\Omega)$, where $i$ is  admissible in $Q^0(C,\Omega)$. If
$M$ is an indecomposable locally free $H$-module and is not isomorphic to $S_{i}$, then
$F_i(M)$ is indecomposable and
\[\underline{\textup{rank}}(F_{i}(M))=s_{i}(\textup{\underline{rank}}(M)).\]
\end{lem}

\begin{prop}\cite[Proposition 9.6]{[GLS1]} \label{prop:TLFreeRig}
Let $M$ be a rigid $\tau$-locally free $H$-module and let $i$ be admissible in $Q^0$.
 Then $F_i(M)$ is also rigid and $\tau$-locally free.
\end{prop}

\begin{prop}\cite[Proposition 1.9]{[DR]} \label{prop:roots} Let $\mathbf{i}=(i_1, \dots, i_n)$ be a
$+$-admissible sequence with respect to the orientation $\Omega$.
The set of  positive roots determined by the Cartan matrix $C$ is  the disjoint union of the following.
\begin{itemize}
\item[(1)] $\{c_\mathbf{i}^{-r}(\beta_{\mathbf{i},k})\mid r\in\mathbb{Z}_{\geq 0},1\leq k\leq n\}.$

\vspace{1mm}

\item[(2)] $\{c_\mathbf{i}^{s}(\gamma_{\mathbf{i},k})\mid s\in\mathbb{Z}_{\geq 0},1\leq k\leq n\}.$

\vspace{1mm}

\item[(3)] $\{x+ r\delta\mid x=0 \textup{ or a positive root that is }<\delta \textup{ and can be deduced
from a certain list of roots}; \\ r\in\mathbb{Z}_{\geq 0}  \textup{ and } r\not=0  \textup{ when } x=0  \}.$
\end{itemize}
\end{prop}

\begin{rem}\label{remDR}
(1) By Proposition \ref{prop3.3}, we know the indecomposable modules at the bottom of the tube of rank $n-1$. Their rank vectors are pairwise distinct and are exactly those in the  list of roots in Proposition \ref{prop:roots} (3) when the orientation $\Omega$ is linear, i.e.
$ \Omega=\{(2, 1), (3, 2), \dots, (n, n-1)\}.$
These rank vectors are $(*)$: the simple roots $\alpha_i$ for $1<i<n$ and
$\sum_{i=1}^n\alpha_i.$
In this case the roots $x$ in Proposition \ref{prop:roots} (3) are sums of the form
\[ \sum_{i\leq t\leq i+j} c_{\mathbf{i}}^t \alpha_2  \]
for some $i, ~j$ with $0\leq i <n-1$ and $0\leq j<n-2$
(see the discussion between Lemma 1.8 and Proposition 1.9 in \cite{[DR]}), where $\mathbf{i}=(n, n-1, \dots, 2, 1)$.
In fact in the sum, $\alpha_2$ can be replaced by any root in the list $(*)$.

\vspace{1mm}

(2) Our main result  of this section below,
Theorem \ref{thm4.1}, largely follows from Theorem \ref{thm3.2} and Proposition \ref{prop:roots}
when $\Omega$ is linear. However, when it is not linear,
Dlab-Ringel do not  explain further how to deduce $x$ from list $(*)$ of roots in the paper \cite{[DR]}.
In our proof to Theorem  \ref{thm4.1}, we will deal  with the quiver $Q$ with nonlinear  orientation
separately, using reflection functors.
\end{rem}

By Lemma \ref{lem:band},
 any band $w$ is equivalent to a band of the form
 \[w_0^{-1}\varepsilon_n^{\pm} w_0 \varepsilon_1^{\pm}...w_0\varepsilon_1^{\pm} \;\;\;\;~~~~~~~~~(**).\]
 We define the {\em delta-length}
 of $w$ by the number $m$ of $w_0$ appearing in the band $(**)$, denoted by
 $\dl(w)=m$. For instance, \[\dl(w_0^{-1}\varepsilon_n w_0 \varepsilon_1)=1\] and
 \[\dl(w_0^{-1}\varepsilon_nw_0\varepsilon_1^{-1}w_0^{-1} \varepsilon_n w_0\varepsilon_1)=2.\]
  If $\dl(w)=r$, $S=(K^s, \varphi)$ is a simple representation of $K[T, T^{-1}]$,
 then the band module $M(w, s, \varphi)$ has rank vector $sr\delta$.

Note that when $C$ is of type $\tilde{C}_{n-1}$ and $D$ is minimal,
 the quiver  $Q=Q(C,  \Omega)$ constructed in \cite{[GLS1]} is exactly the quiver we have in Section 3.2,
$$ \xymatrix{
\varepsilon_1 \circlearrowleft  1 \ar@{-}[r] & 2\ar@{-}[r]&  \cdots \ar@{-}[r]  &n\circlearrowleft \varepsilon_n, \\
}$$
and the algebra $H=H(C,D,\Omega)=KQ/I$, where $I$ is generated by $\varepsilon_i^2$ for $i=1, n$.
We restate  Conjecture 1 for this case as follows.

\begin{conj} \label{conjspecial}
{\em Let $C$ be a Cartan matrix of type $\tilde{C}_{n-1}$ and let  $D$ be a  minimal symmetrizer of $C$. Then
There is a bijection between positive roots of type $C$  and rank vectors of $\tau$-locally free $H$-modules.}
\end{conj}

\begin{thm}\label{thm4.1} Let $H=H(C,D,\Omega)$ be a string algebra of type $\widetilde{C}_{n-1}$.
Then $\alpha$ is a positive root if and only if there is a
$\tau$-locally free module $M$ such that $\rank M=\alpha$. Moreover,
\begin{itemize}
\item[(1)] if $\alpha$ is a positive real root, then there is a unique $\tau$-locally free $H$-module $M$ (up to isomorphism) such that $\underline{\textup{rank}} M =\alpha$.

\item[(2)] if $\alpha=m\delta$ is a positive imaginary root, then all the following modules have rank vector $\alpha$.
\begin{itemize}
\item[(a)] The  modules in level $m(n-1)$ in the tube of rank $n-1$.
\item[(b)] {The  modules in  level $r$ of the homogeneous tubes $\mathcal{H}_{w, S}$, where
$w\in \overline{\mathrm{Ba}}(H)$ with $\dl(w)=t$, $S=(K^s, \varphi)\in \mathcal{S}$ is a simple
$K[T, T^{-1}]$-module such that $r=\frac{m}{st}$. In particular, $r=m$ when $\dl(w)=1$ and
$s=1$. }
\end{itemize}
\item[(3)] the modules at the bottom of the tube of rank $n-1$ are rigid.
\end{itemize}
\end{thm}

\begin{proof}
 Case I:  the orientation $\Omega$ is linear. We first explain that
(3) is true. By Proposition \ref{prop3.3}, the modules at the bottom of the tube of rank $n-1$ are the simples  $S_i$ ($1<i<n$) and $M((\alpha_{21})_{-})$. The simples are rigid since there is no loops at vertices $2, \dots, n-1$, and
$M((\alpha_{21})_{-})$ is rigid, by the homological interpretation of the Ringel Form  defined for $Q$ in \cite[Proposition 4.1]{[GLS1]}.

Next by Lemma \ref{lem4.3} and Remark \ref{remDR}, the roots
in Proposition \ref{prop:roots} (1) are the rank vectors of indecomposable preprojective
modules; the roots
in Proposition \ref{prop:roots} (2) are the rank vectors of indecomposable preinjective
modules; the roots
in Proposition \ref{prop:roots} (3) are the rank vectors of indecomposable modules in
tubes. Therefore the theorem follows from Theorem \ref{thm3.2} and the descriptions
of tubes in Theorem \ref{thm3.1}.

Observation $(\dagger)$:  for a $\tau$-locally free module $M$,  $\rank M$ is
an imaginary root if and only if $M$ is  in a homogeneous tube or in
levels $r(n-1)$ ($r\in \mathbb{N}$) in the tube of rank $n-1$.

Case II: the general case. First
note that any quiver $L'$ of type $A_n$  can be obtained by applying a sequence of admissible reflections $s_{i_1}, \dots, s_{i_m}$ on the linear quiver $L$ of type $A_n$, where $i_{1}$ is admissible in $L$,  $i_{t}$ is admissible in
 $s_{i_{t-1}}\dots s_{i_1} (L)$ for $t>1$ and $L'=s_{i_m}\dots s_{i_1} (L)$.
Assume that the theorem holds for an orientation $\Omega$.
Let $i$ be an admissible vertex in $Q^0(C, \Omega)$. By induction,
we only need to show that
 the theorem holds for the orientation $s_i(\Omega)$.

Let $M$ be an $H$-module  at the bottom of  the tube of rank $n-1$.
Then $M$ is  rigid and $\tau$-locally
free, and by Proposition \ref{prop3.3}, $M$ is not a simple module associated to an admissible vertex. So $F_{i}(M)\not= 0$ is indecomposable by Lemma \ref{lem4.2}. By Proposition \ref{prop:TLFreeRig}, $F_{i}(M)$ is  a
rigid $\tau$-locally free  $H'$-module,
where $H'=H(C, D, s_i(\Omega))$.
Without loss of generality, we assume that $i$ is a sink.  We choose
 a $+$-admissible sequence $\mathbf{i}=(i_1, \dots, i_n)$ with $i_1=i$.
 By Lemma \ref{lem4.2}, $\rank F_{i}^{+}(M)=s_{i}(\rank M)$, which is  not a
  root as those listed in Lemma \ref{lem4.0}.
Note also{ \[\Sigma_{j=0}^{n-2}c_{\mathbf{i}}^j(\underline{\text{rank}}(F_{i}^{+}(M)))=\delta,\]
which is obtained by  applying $s_{i}$ to  $\Sigma_{j=0}^{n-2}c_{\mathbf{i}}^j((\underline{\text{rank}}(M))=\delta$.}
Therefore, $F_{i}^{+}(M)$ is an $H'$-module at the bottom of the tube of rank $n-1$.
Consequently, (3) holds and
 the reflection $s_{i}$ induces a bijection between the rank vectors
of the  $\tau$-locally free $H$-module  in the tube of rank $n-1$ and the rank vectors
of the $\tau$-locally free $H'$-module  in the  tube of rank $n-1$.
Therefore, together with Lemma \ref{lem4.0},
\[\begin{array}{ll}
& \{\rank M\mid M \textup{ is a } \tau \textup{-locally free } H'\textup{-module}\} \\=&\vspace{1mm}
\{s_i(\rank M)\mid M \textup{ is a } \tau \textup{-locally free }
H\textup{-module such that }  \rank M\not =\alpha_i\} \cup \{\alpha_i\}.
\end{array}\]
 The latter is exactly the set of positive roots by the induction hypothesis and the fact that
 $s_i$ permutes $\Delta^{+}\setminus \{\alpha_i\}$.
Therefore $\alpha$ is a positive root if and only if $\alpha=\rank M$ for some
$\tau$-locally free $H'$-module $M$.
The remaining parts of the theorem, (1) and (2),  follow from Theorems \ref{thm3.1}, \ref{thm3.2} and
the observation $(\dagger)$. Therefore,
the theorem holds for $s_i(\Omega)$. This completes the proof.
\end{proof}

\begin{cor}\label{cor4.4}
Let $C$ be a Cartan matrix of type $\widetilde{C}_{n-1}$ and  let $D$ be the minimal symmerizer $\textup{diag}(2,1,1,\cdots,1,1,2)$. Then Conjecture \ref{conjspecial} is true.
\end{cor}

Following Theorem \ref{thm4.1}, we can now enhance Proposition \ref{prop:roots} as follows.

\begin{cor} (cf. \cite[Proposition 1.9]{[DR]})\label{prop4.1}
Let $C$ be the Cartan matrix of type $\widetilde{C}_{n-1}$, $D$ the minimal symmetrizer and let $\mathbf{i}=(i_1,i_2,\cdots,i_n)$ be a $+$-admissible sequence for $(C,\Omega)$. Then
$$\Delta^+(C)=\{c_\mathbf{i}^{-r}(\beta_{\mathbf{i},k})\mid r\in\mathbb{Z}_{\geq 0},1\leq k\leq n\}\cup\{c_\mathbf{i}^{s}(\gamma_{\mathbf{i},k})\mid s\in\mathbb{Z}_{\geq 0},1\leq k\leq n\}\cup$$
{
\[\{(\sum_{p\leq j\leq p+q} c_\mathbf{i}^j(\alpha))+m\delta \mid 0\leq p< n-1, ~0 \leq q< n-2\textup{ and } m \in \mathbb{Z}_{\geq 0} \}\cup \mathbb{Z}_{> 0}\delta,\]}
where $\alpha=\alpha_1+\alpha_2$ (or any other $\alpha_i+\alpha_{i+1}$) if $Q^0$ is alternating, i.e. each vertex is admissible, and otherwise   $\alpha$ can be  any simple root $\alpha_l$ that is associated to
a non-admissible vertex $l$.
\end{cor}

\vspace{2mm}\noindent {\bf Acknowledgements}  The authors would like to thank  Bernt Tore Jensen  for helpful discussions and for pointing out the reference \cite{[BR]}. They also would like to thank  Xiao-Wu Chen and  Zhiming Li for helpful comments.

\end{document}